\documentclass[12pt]{amsart}

\usepackage[margin=1.2in]{geometry}

\usepackage{amssymb, xypic, amsmath, amscd,amsthm,mathrsfs, color}
\usepackage{amsxtra}
\usepackage[pdftex]{graphicx}
\usepackage{enumitem}
\usepackage[normalem]{ulem}
\xyoption{all}
 \usepackage{hyperref}
 \usepackage{color,cancel}
 \usepackage{soul}

\newtheorem{thm}{Theorem}[section]
\newtheorem{prop}[thm]{Proposition}
\newtheorem{lem}[thm]{Lemma}
\newtheorem{cor}[thm]{Corollary}

\theoremstyle{definition}
\newtheorem{definition}[thm]{Definition}
\newtheorem{example}[thm]{Example}

\theoremstyle{remark}
\newtheorem{remark}[thm]{Remark}

\numberwithin{equation}{section}

\def\ri{\rightarrow}

\def\id{\textnormal{id}}
\def\ev{\textnormal{ev}}
\def\OGW{\textnormal{OGW}}
\def\wt#1{\widetilde{#1}}

\def\cJ{\mathcal{J}}

\def\eps{\epsilon}
\def\prt{\partial}
\def\mk#1{\mathfrak{#1}}
\def\co{\colon\!}
\def\fp{\textnormal{fp}}
\def\H{\mathscr{H}}

\def\o{\mathfrak{d}}
\def\sD{\mathscr{D}}
\def\fc{\mathfrak{c}}

\newcommand{\C}{\mathbb{C}}  
\newcommand{\CP}{\mathbb{C}\mathbb{P}}  
\newcommand{\RP}{\mathbb{R}\mathbb{P}}
\newcommand{\Z}{\mathbb{Z}}  
\newcommand{\Q}{\mathbb{Q}}  

\newcommand{\bp}{\bar{\partial}} 
\newcommand{\M}{\widetilde{\mathcal{M}}}
\newcommand{\R}{\mathbb{R}}
\newcommand{\mf}{\mathfrak{M}}
\newcommand{\Lt}{\Lambda^{\text{top}}}
\newcommand{\ov}[1]{\overline{#1}}

\DeclareMathOperator{\ind}{Ind}
\DeclareMathOperator{\aut}{Aut}

\DeclareMathOperator{\Hom}{Hom}

\begin{document}

\title[Open Gromov-Witten disk invariants] {Open Gromov-Witten disk invariants\\ in the presence of an anti-symplectic involution}
\author[Penka Georgieva]{Penka Georgieva*} \thanks{*Partially supported by   NSF grants DMS-0605003 and DMS-0905738.}
\address{Department of Mathematics, Princeton University, Princeton, NJ 08544}
\email{pgeorgie@math.princeton.edu}

\begin{abstract} For a symplectic manifold with an anti-symplectic involution having non-empty fixed locus, 
we construct a model of the moduli space of real sphere maps out of moduli spaces of decorated disk maps and give an explicit expression for its first Stiefel-Whitney class. As a corollary, we obtain a large number of examples, which include all odd-dimensional projective spaces and many complete intersections, for which many types of real  moduli spaces are orientable. For these manifolds, we define open Gromov-Witten invariants with no restriction on the dimension of the manifolds or the type of the constraints if there are no boundary marked points;
  a WDVV-type recursion obtained in a sequel computes these invariants for many real symplectic manifolds.
 If there are boundary marked points, we define the invariants under some restrictions on the allowed boundary constraints, even though the moduli spaces are not orientable in these cases. 
\end{abstract}

\maketitle

\tableofcontents
\section{Introduction}\label{sec_intro}
 The theory of $J$-holomorphic maps
introduced by Gromov \cite{Gr} plays a central role in the
study of symplectic manifolds. Considerations in theoretical physics led to the
development of the Gromov-Witten
invariants. They are invariants of a symplectic manifold~$M$ and can be interpreted
as  counts of $J$-holomorphic maps from a
closed Riemann surface passing
through prescribed constraints. Open String Theory motivated the study of
$J$-holomorphic maps
from a bordered Riemann surface with boundary mapping to
 a Lagrangian submanifold~$L$ and predicts the existence of open
Gromov-Witten invariants. Their mathematical definition, however, has proved to
be
a subtle point and few
   results have been obtained in this setting.  Katz and Liu \cite{KL, Liu} define such invariants for an $S^1$-equivariant pair~$(M,L)$. Without an $S^1$-action, these invariants have been constructed only in dimension 6 or less.  For example, 
    in dimensions 4 and 6 Welschinger  defines real invariants \cite{Wl, Wel} in the case when $L$ is the fixed locus of an anti-symplectic involution and only point constraints are used and   disk invariants \cite{Wel11,Wel12}  under certain  conditions on the homology of~$L$. 
  Also in the case when $L$ is the fixed locus of an anti-symplectic involution  and of dimension 2 or 3,  Cho \cite{Cho} and    Solomon \cite{Sol} define open Gromov-Witten invariants with point constraints; they can be viewed as a Gromov-Witten realisation of Welschinger's real invariants \cite{Wl, Wel}. This   allowed the extension of the definition to a larger class of manifolds and the use of Gromov-Witten techniques for computations leading to a proof of the real mirror symmetry conjecture in \cite{psw}.  In this paper,  we combine the points of view of  \cite{Wl, Wel}  and  \cite{Cho, Sol}  to define disk/real invariants in higher dimensions and with different types of constraints.

    \begin{remark}
The invariants in this paper are defined only in the case when the Lagrangian submanifold is orientable. Thus, they do not recover the results of  \cite{Sol, Wl, Wel} for non-orientable Lagrangians. Furthermore, in the case boundary marked points are used, the invariants in this paper require the use of a non-trivial constraint from the moduli space of domains. Thus, they do not recover the invariants of \cite{Cho, Sol, Wl, Wel} in these cases as well.
   As in \cite{Cho, Sol, Wl, Wel}, we impose conditions preventing the appearance of codimension 1 sphere bubbling strata. The contributions from such boundary strata can be canceled using moduli spaces of maps from $\CP^1$ with a non-standard involution intertwined with the involution on the target. This is done in \cite{Teh}, where the moduli spaces of such maps are studied.
\end{remark}

If $\tau$ is an anti-symplectic involution  on a symplectic manifold~$M$, the fixed locus~$M^\tau$ is a Lagrangian submanifold. In Section \ref{sec:constr}, we construct a new moduli space $\M_{k,l+1}(M,A)$ by  using the anti-symplectic involution~$\tau$ to glue together the boundaries of several moduli spaces  of {\it decorated} disk maps. By Theorem \ref{nb},  this space    is isomorphic to the moduli space of real sphere maps $\ov{\R\mk M}_{k,l+1}(M,A)$ defined in Section~\ref{transv_sec}.  Theorem \ref{gc}       provides an expression    for its first Stiefel-Whitney class, which leads to a large class of examples for which the moduli space is orientable; see  Theorem \ref{cor_rs}.    Proposition \ref{cor_dmor}  completely specifies   which moduli spaces of domains are orientable. We define open Gromov-Witten invariants for strongly semi-positive manifolds satisfying the conditions of Theorem \ref{cor_rs} with no restriction on the dimension of the manifolds or the type of the constraints if there are no boundary marked points; see Theorem~\ref{thm_ogw0}.
  A WDVV-type relation for these invariants sufficient for computing them in many cases, as well as, vanishing and non-vanishing results  are obtained in~\cite{GZ2}. In the case of~$\CP^{2n-1}$, these numbers provide a lower bound for the number of real rational curves which is shown to be sharp in many cases; see \cite{Jan}.
 In the case when   boundary marked points are used, we place a requirement on the allowed boundary constraints only and define invariants for this class of constraints; see Theorem \ref{thm_ogwk}.  Example~\ref{bpts_ex} demonstrates the non-triviality of these invariants for $\CP^{2n-1}$. When both   the invariants of \cite{Cho, Sol, Wl, Wel}   and ours   are defined, the   numbers differ by a known multiple.   We    furthermore show that  the invariants depend non-trivially on the orienting choices   and   give a refined invariant which is often independent of these choices; see Section \ref{ssec_doc}. \\

Throughout this paper,     $(M,\omega)$ is a symplectic manifold, $\tau\co M\ri M$ is an anti-symplectic involution, 
$L=M^{\tau}\subset M$ is the fixed locus of $\tau$ ($L$ is a Lagrangian submanifold), and   $J$ is a   compatible almost complex
structure on $M$ such that $\tau^*J=-J$. We work only in the genus zero case, which means all maps have domains either a sphere or a disk. We fix either  
\begin{alignat}{2} \label{eq_htg}\H_2(M)&=H_2(M;\Q)&~~ \text{and} ~~\H_2(M,L)&=\text{Im}(H_2(M,L;\Z)\ri H_2(M,L;\Q))\quad \text{or}\\ 
\label{eq_hlg}\H_2(M)&=\pi_2(M)&~~ \text{and} ~~ \H_2(M,L)&=\text{ab}(\pi_2(M,L)),
\end{alignat} 
where $\text{ab}(\pi_2(M,L))$ is the abelianization of $\pi_2(M,L)$. In both cases, there is a natural doubling map 
\begin{equation}\label{eq_dmap}\o\co\H_2(M,L)\ri \H_2(M);
\end{equation}
see Section \ref{sec:constr}.
The results of this paper hold for either choice, with the defined invariants in the former case being the obvious sum of  invariants defined for the latter choice.\\
  
 We use the method of $(J,\nu)$-holomorphic maps adapted to the real and bordered settings to achieve transversality as in \cite{RT1, RT2}.
The set $\cJ_\R$ of allowable pairs $(J,\nu)$, defined in Section \ref{transv_sec}, is infinite-dimensional and contractible.  
  We denote the   moduli space of real $(J,\nu)$-holomorphic maps in the class $A\in \H_2(M)$ with $k$ real and $l$ pairs of ordered conjugate points     by $\ov{\R\mk M}_{k,l}(M, A)$ and the moduli space of bordered $(J,\nu)$-holomorphic maps, with boundary mapping to $L$, in the class $b\in \H_2(M,L)$ with~$k$ boundary and $l$    interior points     by $\ov{\mk M}_{k,l}(M,b)$;  see Section \ref{transv_sec}. A (possibly nodal) real $(J,\nu)$-holomorphic  map is called \textsf{semi-simple} if its restriction to the unstable components of its domain is a simple $J$-holomorphic map;  a related notion, called $\tau$-semi-simple, is introduced in Section \ref{transv_sec} for bordered $(J,\nu)$-holomorphic   maps.  Let 
 $$\ov{\R\mk M}_{k,l}^*(M,A)\subset \ov{\R\mk M}_{k,l}(M, A)\qquad \text{and} \qquad \ov{\mk M}_{k,l}^*(M,b)\subset \ov{\mk M}_{k,l}(M,b)$$ 
 denote the possibly non-compact subsets   of ($\tau$-)semi-simple maps. For a generic $(J,\nu)$, the moduli spaces of ($\tau$-)semi-simple maps are cut transversely;  see Section \ref{transv_sec}.
  Let
  $$\ov{\R\mk M}_{k,l}= \ov{\R\mk M}_{k,l}(\text{pt}, 0)\quad \text{and} \quad \ov{\mk M}_{k,l}=\ov{\mk M}_{k,l}(\text{pt}, 0)
  $$
  be the moduli spaces of real and bordered domains,
  respectively, 
   and 
  \begin{align*}&\ev:  \ov{\R\mk M}_{k,l}(M,A)\rightarrow L^k\times M^l  \quad\,\qquad  \mk{f}: \ov{\R\mk M}_{k,l}(M,A)\rightarrow\ov{\R\mk M}_{k,l}\\
&\ev:  \ov{\mk M}_{k,l}(M,A)\rightarrow L^k\times M^l    \qquad\qquad  \mk{f}: \ov{\mk M}_{k,l}(M,A)\rightarrow\ov{\mk M}_{k,l}
\end{align*}
be the   natural evaluation and forgetful    maps,   respectfully.  \\
 
 We   construct the new moduli space $\M_{k,l+1}(M,A)$ 
 in the case  when the domains 
 have at least one interior marked point $z_0$; the latter is signified by the notation $l\!+\!1$.
  We introduce \textsf{decorations} of $\pm$ on the interior marked points and construct the new moduli space using the moduli spaces of  \textsf{decorated}  disk maps in a class $b$ which doubles to $A$ via (\ref{eq_dmap}); see Section \ref{sec:constr}. The decorations are used to give a 1-1 correspondence between decorated marked disks and marked real spheres. The first interior marked point $z_0$ has a special role and is  always decorated with a $+$. We use this point to determine an order on the boundary strata comprised of two bubbles and define an involution $g$ on this boundary using the anti-symplectic involution $\tau$. The new moduli space $\M_{k,l+1}(M,A)$ is   obtained by identifying the boundaries of the decorated moduli spaces via~$g$.
If $k=0$,  we take the class $A\in \H_2(M)$ in the restricted set
$$
\mathcal{A}= \{
A\in \H_2(M)| \, A=\o([u\co(D^2,\prt D^2)\ri (M,L)])\Rightarrow [u_{|\prt D^2}]\neq 0 \in \pi_1(L) \}
$$
in order to 
prevent sphere-bubbling in codimension 1, on which $g$ is not defined.
For example,  all odd-degree classes in $\CP^n$ belong to $\mathcal{A}$. The restriction on the class~$A$  
 is removed in \cite{Teh} by also considering real maps with non-standard involution on the domain~$\CP^1$.

\begin{thm}\label{nb} Suppose $M$ is a symplectic manifold with an anti-symplectic involution~$\tau$, $A\in\H_2(M)$, and $k,l\in \Z^{\geq 0}$ with $k+l>0$.   
\begin{enumerate}[label=(\arabic*), leftmargin=*] 
\item \label{it_evm} The evaluation maps 
\begin{align*}
\ev_{x_i}\co \widetilde{\mathcal{M}}_{k,l+1}(M,A) & \rightarrow L, & \quad [u] & \mapsto u(x_i),\\
\widetilde{\ev}_{z_i}\co \widetilde{\mathcal{M}}_{k,l+1}(M,A) & \rightarrow M, & \quad [u]&\mapsto\begin{cases}
u(z_i),&\,\text{if decoration of }\, z_i \,\text{is}\,+,\\
\tau\circ u(z_i),&\,\text{if decoration of }\, z_i \,\text{is}\,-,\end{cases}\end{align*}
 at the boundary and interior marked points, respectfully, are
    continuous. 
\item\label{it_nb} If $k> 0$ or    $A\in\mathcal{A}$,  the  moduli space of  $\tau$-semi-simple maps $\widetilde{\mathcal{M}}_{k,l+1}^*(M,A)$ is a (possibly non-compact) manifold without   boundary for a generic $(J,\nu)\in\cJ_\R$.  It is smooth outside codimension 2.
\item \label{it_iso} There is a natural homeomorphism  
$$
\sD: \widetilde{\mathcal{M}}_{k,l+1}(M,A)\cong \overline{\mathbb{R}\mathfrak{M}}_{k,l+1}(M,A).
$$
\end{enumerate}
 \end{thm}

By Theorem \ref{nb}\ref{it_iso},  the open Gromov-Witten invariant we define counts real curves passing through prescribed constraints. In particular, if there is a single real curve passing through given constraints, we   count it as~$\pm 1$; see Example \ref{ex_crcex}. The reason for this is that the interior marked point $z_0$ determines the half and the decorations  on the rest of the interior marked points determine whether they belong to the same half or not and for only one choice of decorations will the map pass through the prescribed constraints.   In contrast,     the invariants  of \cite{Cho, Sol}      count such a curve with certain multiples, which relate Cho's and Solomon's invariants to Welschinger's invariants;  for Solomon's invariants, this multiple is described in \cite[Section 7]{Sol2}. On the intersection on which   the invariants of \cite{Cho, Sol} and ours  are defined, the numbers are related by these multiples as well. 
\\
 
 We next study the orientability of the new moduli space $\M_{k,l+1}^*(M,A)$ using the approach of \cite{Geo1}.  We calculate the evaluation of the first Stiefel-Whitney class of $\M_{k,l+1}^*(M,A)$ on a loop $\gamma$ in the moduli space, {\it without} imposing any conditions on the  Lagrangian $L=M^\tau$. We start with any choice of trivializations as in  \cite[Proposition~4.9]{Geo1}, which induce an orientation  of the moduli space at a point $u_0$ on~$\gamma$. We then transport these trivializations along the loop $\gamma$;  at each codimension one boundary strata $\gamma$ crosses, these trivializations define an orientation on the target and the domain of the map $g$. We call the sign of $g$ with respect to these orientations the \textsf{relative sign of  $g$} and   compute it in Section \ref{sec_rsc}. 
  In the relative spin or pin$^\pm$ case, one can talk about the \textsf{sign} of $g$; this sign was computed in \cite{FOOO2} in the relative spin case and in \cite{Sol} in the relative pin$^\pm$ case.
 We then describe the first Stiefel-Whitney class of the moduli space evaluated on the loop $\gamma$  as the sum of the relative signs at each codimension one strata we cross plus the difference in the trivializations at the beginning and at the end of the loop.  Theorem \ref{gc}  expresses this sum in terms of characteristic classes. As a consequence, we obtain in Theorem \ref{cor_rs} a criterion under which the moduli space is as close to being orientable as it can possibly be in light of Proposition \ref{cor_dmor}; this motivates the following definition.

\begin{definition} \label{def_tori} A symplectic manifold
$(M,\omega)$ with an anti-symplectic involution $\tau$ is called \hbox{\textsf{$\tau$-orientable}} if  
   $L=M^\tau$ is orientable and there is a complex bundle $E\ri M$  with an involution $\wt\tau$ lifting~$\tau$ such that $w_2 (  2E^{\wt\tau}\oplus TL)=0$ and
   $$4~|~(\mu(2E, 2E^{\wt\tau})+\mu(M,L))\cdot b \qquad\forall~b\in \H_2(M,L)~\text{s.t.} ~b=-\tau_* b.$$
    A choice of a spin structure on $2E^{\wt\tau}\oplus TL$ and  if not all Maslov indices of the pair $(2E\oplus TM, 2E^{\wt\tau}\oplus TL)$ are divisible by~4, a choice of representatives $b_i\in H_2(M,L;\Z)$ for the elements of $H_2(M,L;\Z)/\text{Im}(\id+\tau_*)$,  is called a \textsf{$\tau$-orienting structure} for $M$. 
\end{definition}
By Lemma \ref{lem_orm}, a $\tau$-orienting structure determines an orientation of the relative determinant bundle of the forgetful map 
\begin{equation}\label{eq_indbl} 
\mk f:\M_{k,l+1}^*(M,A)\ri \M_{k,l+1}\equiv\M_{k,l+1}(\text{pt}, 0)
\end{equation}
outside a certain codimension one stratum $U$.
The list of  $\tau$-orientable manifolds     includes    
$$(\CP^{2n-1},\R\mathbb{P}^{2n-1}), \quad(\CP^1\times \CP^1, \R \mathbb{P}^1\times\R \mathbb{P}^1), \quad(Q\times Q, \text{gr}(f)),$$
 where $Q$ is a symplectic manifold with $w_2(Q)=0$ and $\text{gr}(f)$ is the  graph of   a symplectomorphism $f$ on~$Q$ . Theorem~\ref{cor_rs}   also applies to the  complete intersections $X_{n;\bf{a}}\subset \CP^n$ satisfying the conditions of Corollary \ref{cor_cp}.
  The orientability question is further studied in \cite{GZ1} and less restrictive conditions ensuring orientability are obtained; see \cite[Corollary 5.9]{GZ1}. The orientations induced by a $\tau$-orienting structure, however,  appear to be more natural in the light of recursion formulas such as  in \cite{GZ2} and in higher genus as in \cite{GZ3}; see \cite[Remark 2.6]{GZ2} and \cite[Definition 1.1]{GZ3}.

\begin{remark}
Neither a spin nor a relatively spin structure is sufficient for the orientability of the determinant bundle over the glued space $\M_{k,l+1}^*(M,A)-U$. The $\tau$-orientability condition of this paper implies the existence  of a relative spin structure with the needed additional property  ensuring the orientability of this bundle.
This additional property is not used in
\cite{Cho, Sol}
where the non-orientability of the determinant bundle
is cancelled with the non-orientability coming from the  marked points over
the relevant codimension one strata.
  This method, however,   cannot be used in higher dimensions. The orientation induced by a $\tau$-orienting structure differs from the orientation induced by the associated relative spin structure when $c_1(A)$ is not divisible by $4$.
\end{remark}

Studying the sign of $g$ also provides a simpler approach to results about the orientability of the moduli space of domains
obtained in Proposition~5.7 and Corollary~6.2 in~\cite{ch}. In particular, we prove the following proposition in Section \ref{sec_rsc}.

\begin{prop}\label{cor_dmor} Let $k,l\in \Z^{\geq0}$ with $k+l>0$. The moduli space of domains
$$\M_{k,l+1}\cong\overline{\R\mf}_{k,l+1}$$ is orientable if and only if $k=0$ or $(k,l)=(1,0), (2,0)$. If the moduli space  is orientable, it has a canonical orientation.
\end{prop}

We define open Gromov-Witten invariants for compact    $\tau$-orientable manifolds which are      strongly semi-positive; see Definition~\ref{def_ssp}. We call these manifolds \textsf{admissible}.  

\begin{thm} \label{thm_ogw0}Let $(M,\omega, \tau)$ be   admissible and $A\in \mathcal{A}$. 
\begin{enumerate}[label=(\arabic*), leftmargin=*] 
\item \label{it_ori}  A choice of a  $\tau$-orientating structure for  $M $ determines an orientation on the moduli space $\M_{0,l+1}^*(M,A)$.   
\item \label{it_inv} For a generic $(J,\nu)\in\mathcal{J}_\R$ and a choice of a $\tau$-orienting structure, the images of $\M_{0,l+1}(M, A)$ under the   maps $\ev$ and $\ev\times \mk f$ define   elements 
\begin{equation*} 
\OGW_{A,0,l+1}\in H_*(M^{l+1}; \Z)\qquad \text{and}\qquad
\OGW_{A,0,l+1}\in H_*(M^{l+1}\times \M_{0,l+1}; \Z),
\end{equation*}
respectively.
These classes are independent of the choices of a generic \mbox{$(J,\nu)\in \cJ_\R$} and of a strongly semi-positive deformation of $\omega$.
\end{enumerate}
\end{thm}
 Given $h^M_1,..,h^M_{l+1}\in H^*(M;\Z)$ and $h^{DM}\in H^*(\M_{0,l+1}; \Z)$, the corresponding \textsf{open Gromov-Witten invariant} is the number 
 \[\OGW_{A,0,l+1}(h^M_1,..,h^M_{l+1},h^{DM})=(h^M_1,..,h^M_{l+1}, h^{DM})\cdot \OGW_{A,0,l+1}\]   if the total  degree of $h^M_1,..,h^M_{l+1}, h^{DM}$
is   equal to the dimension of the moduli space and~0 otherwise. 
 A complete recursion for   $\OGW_{A,0,l+1}$ of $\CP^{2n-1}$
 is obtained in \cite{GZ2}, where it is also shown  \cite[Corollary 1.4]{GZ2} that if all $h_i$'s are in odd complex codimensions,   the invariants are non-zero and that they vanish in all other cases. In \cite{Jan}, the vanishing is shown to be sharp in many cases i.e. there are  geometric positions of the constraints for which there are no real curves passing through them. The author is not aware of an example for which these   and Welschinger's invariants do not provide the best lower bound for the count of real curves.\\

  If $k>0$, the moduli space $\M_{k,l+1}(M,A)$ is  not orientable, unless $A$ is minimal and $(k,l)=(1,0),(2,0)$, even if $M$ is admissible. Thus, we generally need cohomology classes with   correctly twisted coefficients in order to obtain a well-defined count. In some cases, it is possible to define invariants without such twisting by  using intersection theory and showing that in a one-parameter family we do not cross ``bad" strata. This is the approach taken in \cite{Cho,Sol, Wl, Wel}, where the authors show that
  if  the dimension of the Lagrangian is 2 or 3 and   only point constraints are used, the strata contributing to the first Stiefel-Whitney class of the moduli space are never crossed; these arguments 
  do not apply to higher dimensions and it is not known whether their conclusion   holds. 
   We overcome these restrictions in many cases by using cohomology classes from the moduli space of domains with coefficients in its orientation system or equivalently their Poincare duals. These classes provide the correct twisting over the stable domains in the case of an admissible manifold as evident from Theorem~\ref{cor_rs}. The contribution to the first Stiefel-Whitney class of $\M_{k,l+1}^*(M,A)$ coming from unstable domains is recorded by  the codimension one stratum $U$ having a single boundary marked point on the second bubble (in addition to the node). This means that the only ``bad" stratum is~$U$ and   showing that it is not crossed is far less restrictive. In particular, this is the case if the minimal Maslov index is greater than the dimension of the Lagrangian,    no matter what   the dimension of the Lagrangian and   the type of  the  constraints are.\\

 If $k>0$ and  $(M, \omega,\tau)$ is admissible, we define the invariant as a signed count of maps passing through prescribed constraints. The sign of a  map~$u$ is determined by the relative orientation of (\ref{eq_indbl})     and is denoted by $\mk s(u)$; see the discussion before the proof of Theorem \ref{thm_ogwk} in Section \ref{sec_ogw}.
     Given $(J,\nu)\in \cJ_\R$ generic and   
 \begin{equation} \label{eq_cc} [A_1],\dots,[A_{l+1}]\in H_*(M;\mathbb{Z}),\quad[B_1],\dots,[B_k]\in H_*(L;\mathbb{Z}), \quad  [\Gamma]\in H_*(\M_{k,l+1}; \mathbb{Z}),
 \end{equation}
 let 
    the number 
    \begin{equation}\label{eq_ogw}\OGW_{A,k,l+1}([A_1],\dots,[A_{l+1}], [B_1],\dots, [B_k], [\Gamma])\in \Z \end{equation}
    be the signed cardinality of the intersection 
    $$(\ev\times\mk{f})(\M_{k,l+1}(M,A))\cap \prod^{l+1}_{i=1}A_i\times\prod^{k}_{j=1}B_j\times \Gamma\subset M^{l+1}\times L^k\times \M_{k,l+1}$$
     if the total codimension of $[A_1],\dots,[A_{l+1}], [B_1],\dots,[B_k],[\Gamma]$ is equal to the dimension of $\M_{k,l+1}(M,A)$ and zero otherwise. Let $|PD(B_j)|$ denote the cohomological degree of the Poincare dual of $[B_j]$.

\begin{thm}\label{thm_ogwk} Let $(M,\omega,\tau)$ be admissible, $A\in\H_2(M)$, $k>0$, $c_{\min}$ be   the minimal Chern number of $(M,\omega)$, and    $[A_i]$, $[B_j] $, and $[\Gamma] $ be as in (\ref{eq_cc}). If $|PD(B_j)|<c_{\min}$ for all $j=1,\dots,k$,  
the number (\ref{eq_ogw})
 is independent of the choice of a regular pair \hbox{$(J,\nu)\in \mathcal{J}_{\R}$}, of a strongly semi-positive deformation of $\omega$, and of the choice of representatives $A_i, B_j$, and $\Gamma$. 
\end{thm}
 
 Under the assumptions of Theorem \ref{thm_ogwk}, we call the number (\ref{eq_ogw}) the \textsf{open Gromov-Witten invariant} corresponding to the constraints (\ref{eq_cc}).
All odd-dimensional projective spaces $\CP^{2n-1}$ with their standard involutions are admissible manifolds with minimal Chern number     $2n$. Thus, the open Gromov-Witten invariant is defined for every choice~of 
$$[A_i]\in H_*(\CP^{2n-1},\Z),\quad [B_j]\in H_*(\RP^{2n-1}, \Z),\quad[\Gamma]\in H_*(\M_{k,l+1},\Z).$$ 
If $k=0$, we may take $[\Gamma]$ to be the fundamental class of $\M_{0,l+1}$; in this case, for $A\in H_2(\CP^{2n-1};\Z)$ odd, $\OGW_{A,0,l+1}$ counts maps constrained only by classes in~$\CP^{2n-1}$. If we     impose boundary constraints, 
we must take a $\Z$-homology class of $\M_{k,l+1}$; see \cite{Cey} for the $\Z$-homology of $\M_{k,l+1}$. By Proposition \ref{cor_dmor}, such classes have positive codimension whenever $k+l>1$  and thus the invariant cannot be interpreted as a count of real curves as the marked domain is constrained geometrically. Example~\ref{bpts_ex} demonstrates the non-triviality of these invariants for $\CP^{2n-1}$. \\

As shown in Example \ref{ex_ntq},  changing the $\tau$-orienting structure   used to define the (local) orientation on $\M_{k,l+1}(M,A)$ sometimes results in a significant change   of the open Gromov-Witten invariants, not just up to a sign. We define a refined invariant by also fixing  the class of the fixed loci of the real maps represent in $H_1(L;\Z_2)$. The absolute value of the refined invariant is independent of the choice of a $\tau$-orienting structure if all Maslov indices of  $(2E\oplus TM, 2E^{\wt\tau}\oplus TL)$ in Definition \ref{def_tori} are divisible by~4. The sum over the classes the real loci represent gives the total invariant; this is discussed  in detail at the end of Section \ref{sec_ogw}. E. Brugall\'e and J. Solomon   communicated to us an additional refinement, related to the refinement  in \cite{IKS}, obtained by fixing a class $d\in H_2(M,L';\Z)/(\id+\text{Im}\, \tau_*)$, for a connected component $L'\subset L$, instead of the class~$A$.   All constructions in this paper respect this refinement. The absolute value of the resulting invariant is independent of the choice of a $\tau$-orienting structure. Note that one may also work on the level of $\pi_2(M,L')/(\id+\text{Im}\, \tau_*)$. The corresponding sums of the absolute values of these invariants give   upper bounds for all other versions of the invariants.  \\

In order to define the invariants for manifolds which are not strongly semi-positive, one needs to use a more sophisticated method to prove the moduli space defines a homology class. With Kuranishi structures, one can  similarly define real and bordered maps using real multi-sections and doubling the bordered map to a real one, respectively; in this way, the involution~$g$ does not have fixed points and the new Kuranishi space defines a homology class when it is orientable. Lemma \ref{lem_orm} implies that for \hbox{$\tau$-orientable} manifolds the tangent bundle of the Kuranishi space, as defined in~\cite{FOOO}, is orientable if there are no  boundary marked points, and thus the image of the moduli space defines a homology class with $\Q$-coefficients. If   there are boundary marked points, however,  we cannot   argue    that the stratum $U$ is still avoided. The approach of~\cite{IP2}, which stabilizes all domains, is expected to remove this restriction. \\
 
The paper is organized as follows.  We    describe  the basic  moduli spaces used in the paper in Section~\ref{transv_sec} and construct the new moduli space  $\M_{k,l+1}(M,A)$ in Section~\ref{sec:constr}. Theorem~\ref{nb}  is proved  in Section~\ref{sec_rmiso}. In Section \ref{sec_rsc}, we compute the relative sign of the map~$g$ and   obtain Proposition \ref{cor_dmor}. The results on the orientability of  $\M_{k,l+1}(M,A)$, including the proofs of Theorem \ref{gc} and Theorem \ref{cor_rs}, are obtained   in Section \ref{sec:or}.  Section~\ref{sec_ogw} is devoted to the definition of the open Gromov-Witten invariants  and the proofs of Theorems \ref{thm_ogw0} and   \ref{thm_ogwk}.     The dependence  of the invariants on the choice of a $\tau$-orienting structure is also discussed in Section \ref{sec_ogw}.\\

The present paper is based on a portion of the author's thesis work completed at Stanford University. The author would like to thank her advisor Eleny Ionel for  her guidance    and encouragement throughout the years. The author would also like to thank  Aleksey Zinger  for many  helpful discussions and for suggesting Example \ref{ex_ntq},
  Erwan Brugall\'e, J\'anos Koll\'ar,  and Jake Solomon for related discussions, and   the referees for their valuable comments.

\section{Preliminaries} \label{transv_sec}

We begin with a description of the real and bordered  moduli spaces used in this paper, adapting the outline in \cite[Section~1]{IP1}. As we only consider the genus zero  case, we   omit the genus from the notation. In the last part of this section, we compare two natural ways of orientating the boundary  of the bordered moduli space at a point and show that the   relation between them depends only on certain topological information.

\subsection{Moduli space of real maps}

  Let  $\widehat{\mathfrak{M}}_n$ denote the space of ordered collections of $n=k+2l\geq3$ marked points on $\CP^1=\C\cup \{\infty\}$. The standard conjugation on $\CP^1$ acts on $\widehat{\mathfrak{M}}_n$. Let $\R\widehat{\mathfrak{M}}_{k,l}$ be the set of   tuples having the first $k$ marked points fixed (real) and each pair $(z_{k+2i-1}, z_{k+2i})$, for $i=1,\dots,l$, invariant under the   conjugation. The order of the $n$-tuple induces orders of the $k$ real points, of the $l$ pairs, and within each pair of conjugate points. Let $\R\mathfrak{M}_{k,l}$ denote the space of equivalence classes of $\R\widehat{\mathfrak{M}}_{k,l}$ modulo the group of real automorphisms  $\aut_\R(\CP^1)$. Its compactification $\overline{\R\mathfrak{M}}_{k,l}$ in the Deligne-Mumford moduli space $\ov{\mk M}_{n}$ consists of equivalence classes of (possibly nodal) stable marked curves $C$, formed as a union of spheres $(\CP^1,j_0)$ attached at nodes and invariant under an involution on~$C$.
  This involution 
    either preserves  a component or interchanges two components.
  An equivalent description of  $\overline{\R\mathfrak{M}}_{k,l}$ in the case   $k>0$ is as the fixed locus of the~$c_{k,l}$ involution on the moduli space $\overline{\mathfrak{M}}_{k+2l}$ given by 
$$
c_{k,l}([j,z_1,\dots, z_{k+2l}])=[-j, z_1,\dots, z_k; (z_{k+2}, z_{k+1}),\dots, (z_{k+2l},z_{k+2l-1})].
$$
 The natural map from the real moduli space of genus $g>0$ marked curves to the fixed locus of the involution on the  complex moduli space is neither injective nor surjective due to the presence of automorphisms; see \cite[Section 4.1]{Liu}, \cite[Section 6.2, 6.3]{Sep}. In genus 0, however, stable curves do not have automorphisms and   the real moduli space is isomorphic to the fixed locus of the involution $c_{k,l}$: each curve in the fixed locus \textsf{has} a \textsf{unique} anti-holomorphic involution compatible with the marked points. The space $\overline{\R\mathfrak{M}}_{k,l}$ coincides with the fixed locus when $k>0$ and is isomorphic to the part of the fixed locus where the involution on the curve has fixed points in the case $k=0$; the part of the fixed locus of~$c_{0,l}$ not covered by $\overline{\R\mathfrak{M}}_{k,l}$ corresponds to curves with real structure without fixed points.\\

 Denote by $\overline{\R\mathcal{U}}_{k,l}$ the universal family   $\overline{\mathfrak{M}}_{k+2l+1}\rightarrow\overline{\mathfrak{M}}_{k+2l}$ restricted to the subspace $\overline{\R\mathfrak{M}}_{k,l}\subset \overline{\mathfrak{M}}_{k+2l}$. We denote by $j_\mathcal{U}$ and $c_\mathcal{U}$ the complex structure and the anti-holomorphic involution on the fibers of $\overline{\R\mathcal{U}}_{k,l}=\overline{\mathfrak{M}}_{k+2l+1}|_{\overline{\R\mathfrak{M}}_{k,l}}\ri  \overline{\R\mathfrak{M}}_{k,l}$ induced by the complex structure on     $\overline{\mathfrak{M}}_{k+2l+1}$ and its anti-holomorphic involution
$$
c_{\mathcal{U}}: [j, z_1,\dots, z_{k+2l+1}])\mapsto[-j, z_1,\dots, z_k; (z_{k+2},z_{k+1}),\dots, (z_{k+2l}, z_{k+2l-1}), z_{k+2l+1}],
$$  
respectively.\\

 A \textsf{real bubble domain $(B, c_B)$ of type $(k,l)$} is a finite union of   oriented $2$-manifolds $B_i$ joined at double points,  with   an orientation-reversing involution $c_B$, with $k$ real and $l$ pairs of conjugate marked points.  Each component $B_i$ is a parametrized sphere with a  standard complex structure $j_0$. The involution $c_B$ either restrict to the standard one on a component or interchanges two components $B_i$ and $B_j$ by the map $z\mapsto \bar{w}$, where $z, w$ are the corresponding coordinates on $B_i, B_j$.
 The components $B_i$, with their special points, are of two types:\\
\hspace*{0.3cm}(a) stable components, and\\
\hspace*{0.3cm}(b) unstable   components which are spheres with at most two special points.
\\
There must be at least one stable component. Collapsing the unstable components to points gives a connected domain $\text{st}(B)$ which is a stable real curve with $k$ real and $l$ pairs of conjugate marked points. Given a real bubble domain $B$, there is   a  unique holomorphic diffeomorphism onto the fiber $\phi_0\co \text{st}(B)\rightarrow \overline{\R\mathcal{U}}_{k,l}$. \\
 
We now define real $(J,\nu)$-holomorphic maps from real bubble domains to a symplectic manifold $(M,\omega)$ with an anti-symplectic involution $\tau$. They depend on a choice of \hbox{$\omega$-compatible} almost complex structure $J$, with $\tau^*J=-J$, and on a real perturbation~$\nu$, chosen from the space of real  sections of the bundle 
\begin{equation}\label{eq_pert}
\Hom (p_2^*(T\overline{\R\mathcal{U}}_{k,l}), p_1^*(TM))\ri M\!\times\! \overline{\R\mathcal{U}}_{k,l},
\end{equation}
 where $p_i$ are the standard projections.

\begin{definition}
A section $\nu$ of the bundle (\ref{eq_pert}) is called a \textsf{real perturbation} if       
$$J\nu=-\nu j_\mathcal{U}\quad \text{and}\quad \nu=d\tau\circ \nu\circ d c_{\mathcal{U}}.
$$
\end{definition}

\noindent
Let \textsf{$\mathcal{J}_\R$} denote the space of pairs $(J,\nu)$ of real perturbations $\nu$ and $\omega$-compatible almost complex structures $J$ on $M$ anti-commuting with $\tau$.   The space of real perturbations is contractible since every section can be symmetrized to satisfy the required conditions. The space of anti-invariant almost complex structures   is non-empty and contractible by the proof of \cite[Proposition~2.63(i)]{MS98}, since it is compatible with the involution~$\tau$. Thus, the space $\mathcal{J}_\R$ is non-empty and contractible.

\begin{definition}\label{def_rjn} Let $(B,c_B)$ be a real bubble domain. 
\begin{enumerate}[label=(\arabic*), leftmargin=*] 
\item A map $u\co B\ri M$ is \textsf{real} if $\tau\circ u\circ c_B = u$.
\item If $(J,\nu)\in \mathcal{J}_\R$, a real map $u\co B\ri M$ is
 \textsf{real $(J,\nu)$-holomorphic } if 
\[ (u,\phi):B \rightarrow M\times \overline{\R\mathcal{U}}_{k,l}\] 
where $\phi = \phi_0 \circ \text{st}$, is a smooth solution of the inhomogeneous Cauchy-Riemann equation 
\[\bp_J u = (u,\phi)^*\nu\]
on each component $B_i$ of $B$.
\end{enumerate}
\end{definition}

In particular, a real $(J,\nu)$-holomorphic map $u\co B\ri M$ is $J $-holomorphic  on each unstable component of~$B$.
A  real $(J,\nu)$-holomorphic map $(u,\phi)$ is called \textsf{stable} if each component $B_i$ of the domain is either stable or its image $u(B_i)$ is nontrivial in $\H_2(M)$.\\

For $A\in \H_2(M)$, let $\mathcal{\R H}^{J,\nu}_{k,l}(M, A)$ denote the set of stable real $(J,\nu)$-holomorphic maps from $(\C \mathbb{P}^1, j_0)$ with $k$ real and $l$ pairs of conjugate   marked points to $M$ with
$[u] = A $. Note that $\mathcal{\R H}^{J,\nu}_{k,l}(M,A)$ is invariant under the  action of $\aut_\R(\CP^1)$: if
$(u, \phi)$ is a real $(J, \nu)$-holomorphic map, then so is $(u\circ \psi, \phi \circ \psi)$ for every   $\psi\in \aut_\R(\CP^1)$.
Similarly, let $\overline{\R\mathcal{H}}^{J,\nu}_{k,l}(M,A)$ be the (larger) set of stable real $(J, \nu)$-holomorphic maps 
from a   real $(k,l)$ bubble domain. Every sequence of real $(J,\nu)$-holomorphic maps from a smooth domain has a sequence that converges modulo   reparametrizations to a stable real $(J,\nu)$-holomorphic map.  
The \textsf{moduli space of real stable maps}, denoted \[\overline{\R\mathfrak{M}}^{J,\nu}_{k,l}(M,A)\qquad \text{or}\qquad\overline{\R\mathfrak{M}}_{k,l}(M,A),\]
is the space of equivalence classes in $\overline{\mathcal{\R H}}^{J,\nu}_{k,l}(M,A)$, where two elements are equivalent if they differ by a real reparametrization.
 As in Section \ref{sec_intro}, let $\overline{\R\mathfrak{M}}_{k,l}^*(M,A)$ denote the    subset of $ \overline{\R\mathfrak{M}}^{J,\nu}_{k,l}(M,A)$ consisting of semi-simple   maps (including from bubble domains). Transversality for simple real $J$-holomorphic maps is described in Section~1 of \cite{Wl} and the method of \cite{RT1, RT2} reduces transversality at stable $(J,\nu)$-holomorphic maps to simple $J$-holomorphic maps. Thus, for  a generic choice of \mbox{$(J,\nu)\in \mathcal{J}_\R$,} the space $\overline{\R\mathfrak{M}}_{k,l}^*(M,A)$ and its strata   are manifolds of expected dimensions.  

\begin{remark} \label{rem_mcr}
  For a generic choice of a real perturbation, all $(J, \nu)$-holomorphic maps from stable domains are  cut transversely. Thus, the complement of $\overline{\R\mathfrak{M}}_{k,l}^*(M,A)$ in $ \overline{\R\mathfrak{M}}^{J,\nu}_{k,l}(M,A)$ consists of multiply covered maps from unstable domains.
  \end{remark}
  
   \subsection{Moduli space of bordered maps}\label{ss_bb}
  
  A \textsf{bordered bubble domain $B$ of type $(k,l)$} is a finite connected union of smooth oriented 2-manifolds $B_i$ joined at double points together with $k$ boundary and $l$ interior marked points, none of which is a double point. Each of the $B_i$ is either a parametrized sphere or a disk with the standard complex structure $j_0$. There must be at least one stable component and the $B_i$, with their special points, fall in    two categories:\\
\hspace*{0.3cm}(a) stable components, and\\
\hspace*{0.3cm}(b) unstable   components which are spheres with at most two special points or disks with at most two boundary  special points.\\

Every bordered $(k,l)$ bubble domain can be doubled to a real $(k,l)$ bubble domain $\hat{B}=B\cup_{\partial B} \bar{B}$, where $\bar{B}$ has the opposite complex structure and the involution $c_{\hat{B}}$ is induced by the identity map $\id\co B\rightarrow\bar{B}$. The boundary   marked points on $B$ become real marked points on $\hat{B}$, and the interior marked points $(z_{k+1},\dots, z_{k+l})$ on $B$ are doubled to the $l$ pairs of conjugate points $((z_{k+i}, c_{\hat{B}}(z_{k+i}))$, with $i=1,\dots,l$.  \\

Every map $u$ from a bordered bubble domain $(B, \partial B)$ to $(M, L)$, where $L$ is the fixed locus of the anti-symplectic involution $\tau$, can be doubled to a real map $\hat{u}\co \hat{B}\rightarrow M$ defined~as 
$$
\hat{u}_{|B}=u\qquad \text{and}\qquad \hat{u}_{|\bar{B}}=\tau\circ u\circ c_{\hat{B}}.
$$
  For $(J,\nu)\in \cJ_\R$,
  a map $u$ from a bordered bubble domain $(B,\partial B)$ to $(M,L)$  is called a \textsf{bordered $(J,\nu)$-holomorphic map}
if the doubled map
\[ (\hat{u},\phi):\hat{B} \rightarrow M\times \overline{\R\mathcal{U}}_{k,l}\] is a real $(J,\nu)$-holomorphic map. 
  A bordered $(J,\nu)$-holomorphic map $(u,\phi)$ is \textsf{stable} if each component $B_i$ of the domain is either stable or its image $u(B_i)$ is nontrivial in     $\H_2(M,L)$.\\
 
For $b \in \H_2(M,L)$, let $\mathcal{H}^{J,\nu}_{k,l}(M,b)$ denote the set of stable bordered $(J,\nu)$-holomorphic maps from $(D^2,j_0)$ with $k$ boundary and $l$ interior marked points to $(M,L)$ with
$[u] = b$. Note that $\mathcal{H}^{J,\nu}_{k,l}(M,b)$ is invariant under the group of conformal automorphisms $G$  of $(D^2,j_0)$: if
$(u, \phi)$ is a bordered $(J, \nu)$-holomorphic  map, then so is $(u \circ \psi, \phi \circ \psi)$ for every   $\psi\in G$.
Similarly, let $\overline{\mathcal{H}}^{J,\nu}_{k,l}(M,b)$ be the (larger) set of stable bordered $(J, \nu)$-holomorphic maps 
from a   bordered $(k,l)$ bubble domain. By \cite[Theorem~3.5]{Fra}, every sequence of bordered $(J,\nu)$-holomorphic maps from a smooth domain has a sequence that converges modulo reparametrization to a stable map.  
The \textsf{moduli space of stable bordered maps}, denoted \[\overline{\mathfrak{M}}^{J,\nu}_{k,l}(M,b)\qquad \text{or}\qquad\overline{\mathfrak{M}}_{k,l}(M,b)\]
is the space of equivalence classes in $\overline{\mathcal{H}}^{J,\nu}_{k,l}(M,b)$, where two elements are equivalent if they differ by a reparametrization.\\

A $J$-holomorphic map $u$ from a bordered bubble domain is called \textsf{$\tau$-simple} if its double~$\hat{u}$ is simple and \textsf{$\tau$-multiply covered} if the double $\hat{u}$ is multiply covered.
   In particular, if a bordered map is not $\tau$-multiply covered, it is $\tau$-simple (this is not the case if~$\tau$ is dropped).  Multiple covers of the disk are $\tau$-multiply covered, but so are maps 
   $$u\!:\!(D^2,\partial D^2) \rightarrow (M,L)\qquad \text{s.t.}\qquad u=\tau\circ u \circ c_{D^2},$$
   where  $ c_{D^2}$ is the standard anti-holomorphic involution on the disk $D^2$;
   see Corollary~\ref{cor_mceqv}.
   Similarly,   a (possibly nodal) bordered $(J, \nu)$-holomorphic map is called \textsf{$\tau$-semi-simple} if its double $\hat{u}$ is semi-simple.  
  Let $\overline{\mathfrak{M}}_{k,l}^*(M,b)$ denote the  subset of $\overline{\mathfrak{M}}^{J,\nu}_{k,l}(M,b)$ consisting  of $\tau$-semi-simple   maps (including from bubble domains). For  a generic choice of $(J,\nu)\in \mathcal{J}_\R$, the space $\overline{\mathfrak{M}}_{k,l}^*(M,b)$ and its strata are   manifolds of expected dimensions, since this is the case for the real space.
For a generic choice of a real perturbation, all $(J, \nu)$-holomorphic maps from stable domains are   cut transversely. Thus,  the complement of $\overline{\mathfrak{M}}_{k,l}^*(M,b)$ in $ \overline{\mathfrak{M}}^{J,\nu}_{k,l}(M,b)$ consists of $\tau$-multiply covered maps from unstable domains.\\

 \subsection{The boundary sign}
In this part, we do not assume  that the Lagrangian submanifold $L\subset M$ is the fixed point locus of an anti-symplectic involution~$\tau$. We assume $J$ is generic.\\
 
For a fibration $F\ri X\overset{\pi}{\ri} B$ there is a canonical isomorphism  
 \begin{equation}\label{eq_fib}
 \Lt X\cong \Lt F\otimes \pi^*\Lt B.
 \end{equation}
 We orient $\mf_{k,l}$ using the canonical orientation on $\mf_{1,1}=$ \{pt\} and the canonical orientations on the fibers of the forgetful maps $\mf_{k,l+1}\ri\mf_{k,l}$ and $\mf_{k+1,l}\ri\mf_{k,l}$. By~(\ref{eq_fib}), there is also an isomorphism 
 \begin{equation}\label{eq_idmf}
 \Lt\mf_{k,l}^*(M,b)\cong \det D\otimes \mathfrak{f}^*\Lt\mf_{k,l},
 \end{equation}
 when $k+l>0$, where $\det D$ is the linearization of $\bp$; see \cite[Section 3.1]{MS98}. If the domain is not stable, we induce an orientation over a point using the fibration $\mf_{k',l'}^*(M,b)\ri\mf_{k,l}^*(M,b)$, for some $k',l'$ such that $k'+l'>0$; see the proof of \cite[Corollary~1.8]{Geo1}. By \cite[Proposition~4.9]{Geo1}, an orientation on $\det D$ over a point $u_0$ is induced by  choices of trivializations of 
 $$\wt{F}=TL\oplus 3\det TL\qquad \text{and}\qquad F^1=\det TL$$
  over $u_0(\prt D^2)$ and $u_0(x_1)$ for some $x_1\in\prt D^2$, respectively. Thus, these choices induce an orientation of $\mf_{k,l}^*(M,b)$ at $u_0$. \\

 We denote by
\begin{equation}\label{eq_fpd}
\mathfrak{M}_{k_1,l_1}(M,b_1)\times_{\fp}  \mathfrak{M}_{k_2,l_2}(M,b_2)\subset \mathfrak{M}_{k_1,l_1}(M,b_1)\times \mathfrak{M}_{k_2,l_2}(M,b_2)
\end{equation}
the subspace of pairs of disk maps with the same value at the last boundary marked point of the first map  and the first boundary marked point of the second map.
 This fiber product appears as a boundary stratum of $ \overline{\mathfrak{M}}_{k,l}(M,b)$, where $b_1+b_2=b$, $l_1+l_2=l$, and $k_1+k_2-2=k$; the node corresponds to the last marked point on each   factor.
  By (\ref{eq_idmf}), there is a canonical isomorphism 
\begin{equation}\label{eq_fp}
\begin{split}
\Lt \big(\mf_{k_1,l_1}^*(M,b_1)&\times_{\fp}\mf_{k_2,l_2}^*(M,b_2)\big)\\
&  \cong \det D_1\otimes \det D_2 \otimes \Lt TL\otimes \mathfrak{f}^*\Lt \mf_{k_1,l_1}\otimes\mathfrak{f}^*\Lt\mf_{k_2,l_2},
\end{split}
\end{equation}
if the domains are stable. 
Orientations of $\det D_1$ and $\det D_2$ at a point $u_1$ of the fiber product are induced by  choices of  trivializations of $\wt F$ over  the image under $u_1$ of the boundary of the domain and of $F^1$ over a point in this image.
If   one of the domains is not stable, the left-hand side of (\ref{eq_fp}) is oriented at $u_1$ by adding an interior marked point to stabilize the domain and then forgetting it; see the proof of \cite[Corollary~1.8]{Geo1}. Thus, the above choices of trivializations determine an orientation of the fiber product at $u_1$, which we call a \textsf{fiber product orientation at $u_1$}.\\

 Let $\gamma=\gamma(t)$ be a path in $\overline{\mf}_{k,l}^*(M,b)$, transversal to the boundary,  with $\gamma(0)=u_0$ and $\gamma(1)=u_1\in\prt\overline{\mf}_{k,l}^*(M,b)$. There is a canonical isomorphism
\begin{equation}\label{eq_ddm}
\Lt\ov{\mf}_{k,l}^*(M,b)_{|\gamma}\cong \det D_{|\gamma}\otimes \mathfrak{f}^*\Lt\ov{\mf}_{k,l |\gamma},
\end{equation}
if the domain of $u_1$ is stable.
 Given an orientation at $u_0$ of  $\ov{\mf}_{k,l}^*(M,b)$, the path~$\gamma$ induces an  orientation at $u_1$ of $\prt \ov{\mf}_{k,l}^*(M,b)$  by 
$$
\Lt\ov{\mf}_{k,l}^*(M,b)_{\gamma(1)}\cong \frac{d\gamma}{dt}|_{t=1}\otimes \prt\ov{\mf}_{k,l}^*(M,b)_{\gamma(1)}.
$$ 
Since  
\begin{align*}
\Lt\ov{\mf}_{k,l}^*(M,b)_{\gamma(1)} &\cong\det D_{\gamma(1)}\otimes  \mathfrak{f}^*\Lt\ov{\mf}_{k,l|\mk{f}(\gamma(1))}\\
&\cong \det D_{\gamma(1)}\otimes \frac{d\gamma}{dt}|_{t=1}\otimes  \mathfrak{f}^*\Lt\prt\ov{\mf}_{k,l|\mk{f}(\gamma(1))},
\end{align*}
\textsf{the boundary orientation at $u_1$} is induced~by
\begin{enumerate}[label=$\bullet$, leftmargin=*]
 \item the orientation on $\det D_{\gamma(1)}$ induced from the orientation on $\det D_{\gamma(0)}$  and
 \item the orientation on the boundary of $\ov{\mf}_{k,l}$ induced from its interior,  twisted by~$(-1)^{\text{ind} D}$.
 \end{enumerate}
 \noindent
If     the domain of $u_1$ is not stable, the left-hand side of (\ref{eq_ddm}) is oriented along $\gamma$ by adding an interior marked point to stabilize all domains and then forgetting it; see the proof of \cite[Corollary~1.8]{Geo1}. Thus, the above 
isomorphisms still determine an orientation of $\prt\ov{\mf}_{k,l}^*(M,b)$ at $u_1$.\\

Given trivializations of $\wt F$ and $F^1$ at $u_0$, we use the path $\gamma$ to transport them to~$u_1$.   If $L$ is not orientable, we assume $k>0$ and use the first boundary marked point to transport the trivialization of $F^1$; if $L$ is orientable, the choice at $u_0$ induces an orientation  of $F^1$ and the assumption on $k$ is not necessary. Thus, these choices of trivializations and the path $\gamma$ induce orientations on the fiber product at $u_1$ and on the boundary of the moduli space  at $u_1$. We call the sign of the inclusion map
$$
\mathfrak{M}_{k_1,l_1}^*(M,b_1)\times_{\fp}  \mathfrak{M}_{k_2,l_2}^*(M,b_2)\hookrightarrow \partial \overline{\mathfrak{M}}_{k,l}^*(M,b), 
$$
at $u_1$
 with respect to these orientations the
 \textsf{  boundary sign at $u_1$} and denote it by~$\eps(u_1)$.
  Suppose we choose different trivializations of $\wt F$ and $F^1$ with differences expressed by $sp\in\pi_1(SO(n+3))$ and $o\in\pi_0(O(1))$. By \cite[Proposition 4.9]{Geo1}, the effect of this change on the orientation of $\det D_{|\gamma(0)}$, and hence on the induced boundary orientation at $u_1$, is  equal to $sp+(\mu(b) +1)o$, where $\mu(b)$ is the Maslov index of $(TM, TL)$ on~$b$. Transporting the new trivializations to $u_1$ induces new trivializations of $\wt F$ on the boundaries of the two bubbles, with differences expressed by $sp_1, sp_2\in \pi_1(SO(n+3))$, and a new trivialization of $F^1$ at the node, with the difference equal to $o$. The change of the orientation on $\det D_1\otimes \det D_2 \otimes F^1_{|\text{node}}$ equals 
  $$sp_1+(\mu(b_1)+1)o+sp_2+(\mu(b_2)+1)o+o=sp+(\mu(b)+1)o,$$
   since $sp=sp_1+sp_2$ and $\mu(b)=\mu(b_1)+\mu(b_2)$. Thus, this change equals the change of the boundary orientation and   the sign between the two is unchanged. Choosing a different path or initial point amounts to similarly changing the trivializations at $u_1$, which induces the same change of the boundary and fiber product orientations. Therefore, the boundary sign at $u_1$ is well-defined.
 
 \begin{lem}\label{lem_sgn} Suppose $L$ is a Lagrangian submanifold of a symplectic manifold $M$ and $u_1$ is an element of the fiber product (\ref{eq_fpd}) for some  $b_i\in \H_2(M,L)$ and \hbox{$k_i,l_i\in\Z^{\geq 0}$} for $i=1,2$.
 The boundary sign $\eps(u_1)$ defined above   depends only on the Maslov indices $\mu(b_1)$ and $\mu(b_2)$, the sign $\eps_{\text{DM}}$ of the inclusion $\mf_{k_1,l_1}\times\mf_{k_2,l_2}\hookrightarrow \prt \ov{\mf}_{k,l}$, and the dimension of $L$.
 \end{lem}

 \begin{proof}
 With notation as before, there is a canonical (up to homotopy) isomorphism
 \begin{equation*}
 \begin{split}
 \frac{d\gamma}{dt}|_{t=1} \otimes \det D_1\otimes \det D_2 \otimes \Lt TL\otimes \mathfrak{f}^*\Lt \mf_{k_1,l_1}&\otimes\mathfrak{f}^*\Lt\mf_{k_2,l_2}\\ 
 &\cong \det D_{|\gamma(1-\rho)}\otimes \mathfrak{f}^*\Lt\ov{\mf}_{k,l |\gamma(1-\rho)},
 \end{split}
 \end{equation*}
 with $\rho\in[0,1]$ and the top exterior products on the left-hand side   taken over ~$u_1$.  This isomorphism is the tensor product of the isomorphisms
\begin{align}
\notag \frac{d\gamma}{dt}|_{t=1} \otimes\Lt \mf_{k_1,l_1}\otimes\Lt\mf_{k_2,l_2}&\cong  \Lt\ov{\mf}_{k,l |\gamma(1-\rho)}\quad
\text{and}\\ 
\label{eq_detgl} \det D_1\otimes \det D_2 \otimes \Lt TL&\cong \det D_{\gamma(1-\rho)}
\end{align}
corresponding to the inclusion $\mf_{k_1,l_1}\times\mf_{k_2,l_2}\hookrightarrow \prt \ov{\mf}_{k,l}$
and the gluing of real Cauchy-Riemann operators as in \cite[Section 3.2]{EES} and \cite[Section 4.1]{WW}, respectively.\\

Thus, we need to show that the sign of the   isomorphism (\ref{eq_detgl}) with respect to the orientations induced by
trivializations as above depends only on $\mu(b_1)$ and~$\mu(b_2)$.
This isomorphism commutes with isomorphisms (over the identity) of bundle pairs  as defined in \cite[Definition C.3.4]{MS98}. The   trivializations of $\wt F$ and $F^1$ at $u_1$, which are used  to orient $\det D_1$ and $\det D_2$, are determined by transporting the  trivializations of $\wt F$ and $F^1$ at~$u_{1-\rho}$,    used to orient $\det D_{\gamma(1-\rho)}$;
   this transport also commutes with fiberwise isomorphism of bundle pairs along the path.
Finally, the procedure of orienting the determinant bundle (on a smooth domain) using  trivializations of $\wt F$ and $F^1$ commutes with   isomorphisms of bundle pairs as well; see the proofs of   \cite[Propositions 3.1, 3.2]{Geo1}.
   Thus, the sign of the isomorphism (\ref{eq_detgl})   depends only on the isomorphism class of the bundle at $u_1$, which depends only on the Maslov indices $\mu(b_1)$ and $\mu(b_2)$. 
 \end{proof}
  
\section{Construction of $\M_{k,l+1} (M,A)$}\label{sec:constr}
 
In this section, 
 we   construct a new moduli space $\M_{k,l+1}(M,A)$ assuming 
  there is at least one interior marked point~$z_0$.
  The construction is carried by gluing together several copies of decorated versions of $\overline{\mathfrak{M}}_{k,l+1}(M,b)$. The decorations on the interior marked points  are   used to provide an order within the conjugate pairs on the double.\\

For $b\in\H_2(M,L)$, let    $\bar{b}=-\tau_*(b)$.
 Let $\o\co\H_2(M,L)\ri\H_2(M)$ be the homomorphism described as follows. With the choice~(\ref{eq_hlg}), we simply take the double of the bordered map as described in Section \ref{transv_sec}. 
With the choice~(\ref{eq_htg}), we define $\o(b)\in H_2(M;\mathbb{Q})$ for $b\in H_2(M,L;\mathbb{Q})$ to be the unique element    which maps to $b-\tau_*(b)$ under the homomorphism $$j_*\co H_2(M;\Q)\rightarrow H_2(M,L;\Q)$$ and satisfies $\tau_*(\o(b))=-\o(b)$. Over $\mathbb{Z}$, this element is not unique and the difference between every two such elements  is torsion; this is the reason we take $\Q$ instead of $\Z$ coefficients.    We observe that $\o(b)=\o(\bar{b})$.

   Let 
 $$
 \widetilde{\mathfrak{M}}_{k,l+1}(M,b)=\overline{\mathfrak{M}}_{k,l+1}(M,b)\times \mathbb{Z}_2^{l}
 $$
  be the moduli space of $(J,\nu)$-holomorphic disk maps which represent a fixed class $b\in \H_2(M,L)$ and have a \textsf{decoration} of $+$ or $-$ on the interior marked points, with the decoration of  $z_0$ being always $+$. This decoration is used to define a 1-1 correspondence between disks with \textsf{decorated} marked points and spheres with ordered pairs of conjugate points.
 Let  
 $$\mathcal{M}_{k,l+1}(M,A)=\displaystyle{\coprod_{\o(b)=A}}\widetilde{\mathfrak{M}}_{k,l+1}(M,b)$$
 be the disjoint union of decorated moduli spaces of disk maps, representing a class $b$ whose image under $\o$ is a fixed class $A\in \H_2(M)$.  \\

In  Section \ref{ss_bb}, we described a doubling construction for marked bordered maps which we now modify  for    decorated bordered maps. The construction of the doubled map $\hat{u}$, the doubled domain $\hat{B}$, and the real marked points is as before, but an  interior marked point $z_i$ with decoration $+$ is now mapped to the pair $\mk p (z_i^+)=(z_i, c_{\hat{B}}(z_i))$, while an interior marked point $z_j$ with decoration $-$ is mapped to the pair $\mk p(z_j^-)=(c_{\hat{B}}(z_j), z_j)$. 
\begin{figure}[htb]
\begin{center}
\leavevmode
\includegraphics[width=0.6\textwidth]{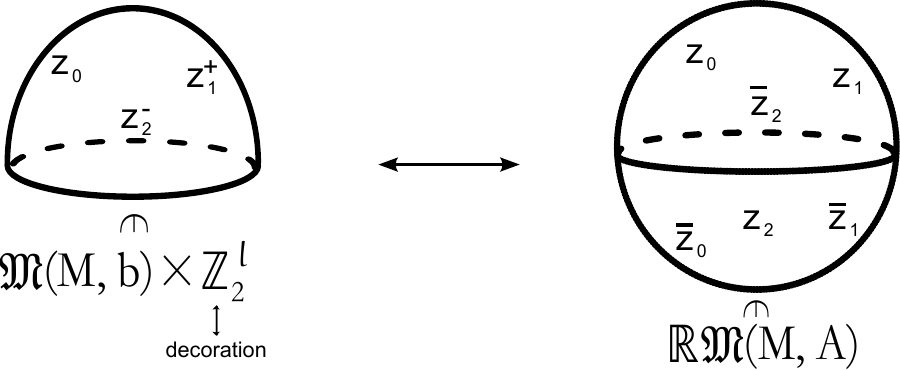}
\end{center}
 \end{figure}

For each bordered bubble domain $B$, denote by $B'$ the  bordered bubble domain obtained from~$B$ by replacing each nodal point $z_{i,j}$ on each component $B_i$ by $\bar z_{i,j}$, but identifying the same pairs of components as in~$B$;   recall that each $B_i$ is either the standard sphere or the standard disk and here $\bar z$ denotes the conjugate of $z$ with respect to the standard conjugation on the sphere or the disk.
Let  $\fc_{B}\co B\rightarrow B'$ be the   anti-holomorphic
   map
  sending $z\in B_i$ to $\bar z\in B_i'=B_i$.
     If $u\co(B,\prt B)\rightarrow (M,L)$, $x_1,\dots, x_k\in\prt B$, and $z_1,\dots, z_l\in B-\prt B$,
 the correspondence 
 $$
 \mathfrak{c}: (B, u, {\bf x},  z_1^\pm,\dots, z_l^\pm)\mapsto (B',\tau\circ u\circ \fc_{B'}, \fc_{B}({\bf x}), \fc_{B}( z_1)^\mp,\dots,  \fc_{B}(z_l)^\mp)
 $$
  is called \textsf{conjugation}. Note that it reverses the decorations.

 \begin{lem}\label{lem_dbl} 
 If $(B,u, {\bf x},  z_1^\pm,\dots, z_l^\pm)$ is a bordered $(J,\nu)$-holomorphic   map with   image $[u]=b\in \H_2(M,L)$, then 
 $(B',\tau\circ u\circ \fc_{B'}, \fc_{B}({\bf x}), \fc_{B}( z_1)^\mp,\dots,  \fc_{B}(z_l)^\mp)$ is a $(J,\nu)$-holomorphic map with image $[\tau\circ u\circ \fc_{B'}]=\bar{b}\in \H_2(M,L)$.
 \end{lem}

 \begin{proof} Let $u'=\tau\circ u\circ \fc_{B'}$ and denote by $\fc_{\hat{B}}\co \hat B\ri \hat B'$ the extension of the involution~$\fc_B$ defined  by 
 $$c_{\hat B'}\circ \fc_{\hat B}= \fc_{\hat B}\circ c_{\hat B}\equiv \psi^{-1}.
 $$  
 In particular, $\psi^{-1}(\mk p(z_i^\pm))= \mk p (\fc_B(z_i)^\mp)$. 
 If 
 $$
 \phi=\phi_0\circ \text{st}\co \big(\hat B,{\bf x}, \mk p (z_1^\pm), \dots, \mk p(z_l^\pm)\big)\ri \mathcal{U}_{k,l}
 $$
 is the holomorphic map onto a fiber as in Definition \ref{def_rjn}, then so is 
  $$
 \phi'\equiv \phi\circ \psi \co \big(\hat B',\fc_{B}({\bf x}), \mk p (\fc_{B}(z_1)^\mp\big), \dots, \mk p( \fc_{B}(z_l)^\mp))\ri \mathcal{U}_{k,l}.
 $$
  Since
 \begin{align*}\hat{u}'&= u'=\tau\circ u \circ \fc_{B'}= \tau \circ u \circ c_{\hat B}\circ c_{\hat B}\circ \fc_{ B'}= \hat u \circ \psi &  \text{on}\,  B', \\
 \hat{u}'&=\tau\circ u'\circ c_{\hat{B}'}= u \circ \fc_{B'}\circ c_{\hat B'}= \hat u \circ \psi & \text{on}\,  \bar B',
   \end{align*}
   $\hat u' = \hat u \circ \psi$. Using $(\hat u', \phi')=(\hat u, \phi)\circ \psi$, we find that 
   $$
   \bp_J \hat u ' = \psi^*\bp_J \hat u = \psi^* (u, \phi)^*\nu= (\hat u', \phi')^*\nu.
   $$
   Thus, $u'$ is $(J,\nu)$-holomorphic. 
\end{proof}

\begin{cor}\label{cor_mceqv}
If $u\co(D^2,\partial D^2)\rightarrow (M,L)$ satisfies $u\circ \varphi=\tau\circ u\circ c_{D^2}$ for some holomorphic $\varphi\co D^2\ri D^2$, then      $u$ is $\tau$-multiply covered.
\end{cor}

\begin{proof}
By the proof of  Lemma \ref{lem_dbl}, the double of $\tau\circ u \circ c_{D^2}$ equals $\hat u\circ \psi$, where $\psi$ is orientation-reversing on the real locus; the double of $u\circ \varphi$ is $\hat u\circ \hat \varphi$, where $\hat \varphi$ is orientation preserving on the real locus. Thus, $\hat u$ is multiply covered.
\end{proof}

\begin{remark} In contrast to \cite[Section 4]{FOOO2} and \cite[Section 5]{Sol}, we do not require $\nu$
  to be equivariant with respect to $c_{D^2}$ and $\tau$, that is $\nu\neq\text{d}\tau\circ \nu\circ \text{d}c_{D^2}$ in our applications; thus, $u\neq \tau\circ u\circ c_{D^2}$ for $u\in \mk M_{k,l+1}(M,b)$. This is necessary in order to obtain transversality over the real moduli space.   Lemma~\ref{lem_dbl} asserts that the {\it doubles} of  $u$ and       $\tau\circ u\circ c_{D^2}$ are $(J,\nu)$-holomorphic for the same real pairs $(J,\nu)$.
\end{remark}
 
We define an involution
$g$ on the compactification of the strata of $\mathcal{M}_{k,l+1}(M,A)$ comprised of maps from a bubble domain having a single boundary node by
 \begin{gather*}
g: \widetilde{\mathfrak{M}}_{k_1,l_1+1}(M,b_1)\times_{fp} \widetilde{\mathfrak{M}}_{k_2,l_2}(M,b_2)\rightarrow \widetilde{\mathfrak{M}}_{k_1,l_1+1}(M,b_1)\!\times_{fp}
\widetilde{\mathfrak{M}}_{k_2,l_2}(M,\bar{b}_2)\\
[(u_1,\vec{x}_1,\vec{z}_1;\sigma_1),(u_2,\vec{x}_2,\vec{z}_2;\sigma_2)]\mapsto [(u_1,\vec{x}_1,\vec{z}_1;\sigma_1),(\tau\circ u_2\circ \fc_{B'}, \fc_{B}(\vec{x}_2), \fc_{B}(\vec{z}_2), 
\bar{\sigma}_2)],
\end{gather*} 
where $\sigma_i\in \Z_2^{l_i}$ records the decorations, $\bar{\sigma}_i$ is the reverse decoration of $\sigma_i$, and \hbox{$\fc_B\co B\ri B'$} is the    anti-holomorphic map defined above. By Lemma~\ref{lem_dbl}, the image of $g$ is indeed an element of $\mathcal{M}_{k,l+1}(M,A)$. The subscript $l_1+1$ is used to specify the bubble carrying    the preferred interior marked point $z_0$. The point $z_0$ is used to define an order on the two disks using which we define the map $g$ unambiguously.  More intuitively, the map~$g$ is given by identity on the first factor and by the conjugation~$\mk c$ on the second, with first and second   determined by $z_0$.
Let
$$
\widetilde{\mathcal{M}}_{k,l+1}(M,A)=\mathcal{M}_{k,l+1}(M,A)/_{u \sim g(u) }
$$
be   the space obtained by identifying the corners of the moduli spaces $\widetilde{\mathfrak{M}}_{k,l+1}(M,b)$ using the map $g$.
  In particular, on the compactification of the  strata with $p$ boundary nodes,  we  can order    the factors of the fiber product, with the first one containing $z_0$,      to define a map $g^p_i$ by $\mathfrak{c}$ on the $i$-th component and identity on the rest, for $i=2,
\dots, p$. If $u$ belongs to such a stratum, $u\sim g_i^p(u)$  for all $i=2,\dots, p$; this is independent of the ordering of the factors. A similar construction for $M=\CP^4, k=0, $ and $l=-1$ is described in detail in the proof of  \cite[Proposition 11]{psw}.\\

\begin{figure}[htb]
\begin{center}
\leavevmode
\includegraphics[width=0.8\textwidth]{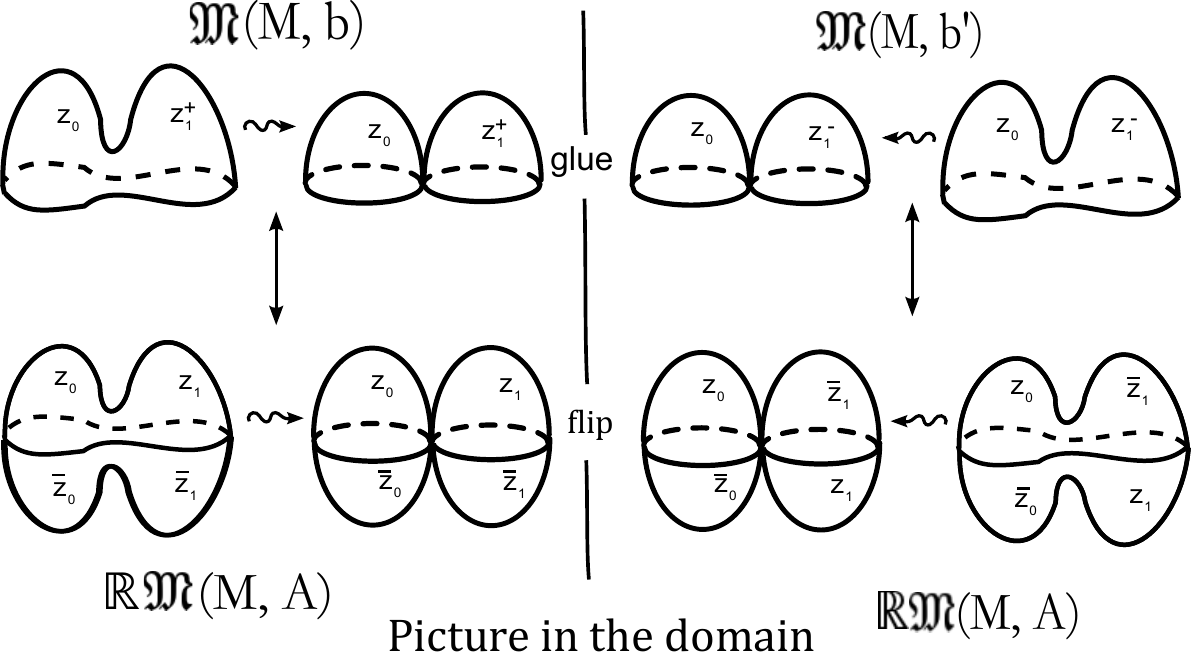}
\end{center}
 \end{figure}

  By taking the target manifold $M$ to be a point in the above construction, we obtain an analogue  of the Deligne-Mumford moduli space, which we denote by $\widetilde{\mathcal{M}}_{k,l+1}$.  There is  again a  forgetful map 
 $$\mk{f}: \widetilde{\mathcal{M}}_{k,l+1}(M,A)\rightarrow \widetilde{\mathcal{M}}_{k,l+1}$$ given by forgetting  the maps and stabilizing the domains. The moduli space $\widetilde{\mathcal{M}}_{k,l+1}$ has no boundary provided there is at least one boundary marked point. Indeed, since  all bubbles are stable, each of them must have either at least $3$ boundary marked points, on the order of which $g$   acts nontrivially, or a boundary marked point and an interior marked point, and $g$   acts nontrivially on the decoration of the interior point.
 If $k=0$, the moduli space of disk domains has a boundary component consisting of marked spheres on which $g$ is not defined;   thus, $\M_{0,l+1}$ has boundary in this case as well. 
 
 \section{Proof of Theorem \ref{nb}}\label{sec_rmiso}

\begin{proof}[\bf \emph{Proof of   \ref{it_evm}}]
    The evaluation maps at a boundary marked point $x_i$ fit together to form an evaluation map on $\M _{k,l+1}(M,A)$, since 
    $$u(x_i)= \tau\circ u\circ \fc_{B'} (\fc_B(x_i)).
    $$
    This does not hold at an interior marked point, since $\tau$ does not fix the image. We account for this by   using the decorations of the interior marked points and define a new evaluation map~by
     \[\widetilde{\ev}_{z_i}(u)=\begin{cases}
u(z_i),& \text{if the decoration of $z_i$ is $+$},\\
\tau\circ u(z_i),& \text{if the decoration of $z_i$ is $-$. }
\end{cases}\]
Since $g$ changes the decoration, the maps $\widetilde{\ev}_{z_i}$ fit together and form an evaluation map at the interior marked points  $\widetilde{\ev}_{z_i}:\M _{k,l+1}(M,A)\rightarrow M$.
\end{proof}

\begin{proof}[\bf \emph{Proof of   \ref{it_nb}}]
 The moduli space $\widetilde{\mathfrak{M}}^*_{k,l+1}(M,b)$    has two possible types of boundary   - one having two disk components on which $g$ is defined and one having a single sphere component on which $g$ is not   defined. The latter is avoided by  requiring that either there be at least one boundary marked point or that  all classes which the boundary of the disk represents in  $\pi_1(L)$ be nonzero. Thus, to prove the statement, we need to show that
the map $g$ has no fixed points on the codimension one strata. Indeed, if $[u_2]=[\tau\circ u_2\circ c_{D^2}]$, then by Corollary \ref{cor_mceqv}, $u_2$ is $\tau$-multiply covered and thus $u=(u_1,u_2)$ is not an element of $\widetilde{\mathcal{M}}_{k,l+1}^*(M,A)$. \end{proof}
 
  \begin{proof}[\bf \emph{Proof of \ref{it_iso}}] 
 By definition, a bordered map  $u$ is $(J,\nu)$-holomorphic if its double $\hat{u}$ is $(J,\nu)$-holomorphic. Along with the doubling construction before Lemma~\ref{lem_dbl}, this defines a map
  $$
  \coprod_{\o(b)=A}\overline{\mathcal{H}}_{k,l+1}(M, b)\times \Z_2^l\rightarrow \overline{\mathcal{H}}_{k,l+1}(M, A).
  $$
   Since every reparametrization  of the bordered domain defines a real reparametrization of the doubled domain, this map descends to a map 
   $$\coprod_{\o(b)=A}\widetilde{\mathfrak{M}}_{k,l+1}(M,b)\rightarrow \overline{\R\mathfrak{M}}_{k,l+1}(M, A).
   $$
 By the proof of Lemma \ref{lem_dbl}, the classes of the doubles of 
 $$[u,{\bf x},{\bf z},\sigma]\qquad \text{and} \qquad[\tau\circ u \circ \fc_{B'}, \fc_{B}({\bf x}),\fc_{B}({\bf z}),\bar{\sigma})]$$
  are equal in $\overline{\R\mathfrak{M}}_{k,l}(M, A)$, as they differ by the real reparametrization $\psi=c_{\hat B}\circ \fc_{\hat B'}$. 
 Therefore,  the last map further descends to a continuous map 
   \begin{align*} \sD: \widetilde{\mathcal{M}}_{k,l+1}(M,A)&\rightarrow \overline{\mathbb{R}\mathfrak{M}}_{k,l+1}(M,A) \\
  \big[u, B, {\bf x}, z_0^+, z_1^\pm\dots, z_l^\pm\big]&\mapsto \big[\hat{u},\hat{B}, {\bf x}, \mk p(z_0^+),\mk p(z_1^\pm)),\dots ,  \mk p (z_l^\pm))\big].
  \end{align*}
 We show that   $\sD$ is both surjective and injective.\\
\\
  {\bf Surjectivity.} Let $[\hat u,{\bf x}, (z_0, c_{\hat B}(z_0)),\dots, (z_l, c_{\hat B}(z_l))]$  be an element of the real moduli space $\ov{\R\mf}_{k,l+1}(M,A)$. Pick a representative $\hat u: \hat B\ri M$ and choose a half $B\subset \hat B$ such that $\hat B=B\coprod \bar{B}$ and the first element $z_0$ of the $0$-th conjugate pair of interior marked points belongs to $B$. Then, 
  $$
\sD( [\hat u_{|B}, {\bf x}, r(z_0),\dots, r(z_l)])=[\hat u, {\bf x}, (z_0, c_{\hat B}(z_0)),\dots, (z_l, c_{\hat B}(z_l))],
$$
where
  $r(z_j)$ denotes the correspondence
  \[r(z_j)=\begin{cases}(z_j,+),&\text{if $z_j\in B\subset \hat{B}$},\\ (c_{\hat{B}}(z_j),-),&\text{if $z_j\in \hat{B}-B$}.\end{cases}
 \]
 \\
 {\bf Injectivity.}   Let 
 $$[u_i, {\bf x}_i, z_{i;0}^+, z_{i;1}^\pm,\dots z_{i;l}^\pm ] \quad \text{for} ~i=1,2$$ be two elements of $\M_{k,l+1}(M,A)$ with representatives $u_i\co (B_i,\prt B_i)\ri(M,L)$ and doubles 
 $$(\hat{u}_i,{\bf x}_i, \mk p (z_{i;0}^+),\mk p( z_{i;1}^\pm),\dots \mk p(z_{i;l}^\pm) ).
 $$
  Suppose there is a real holomorphic diffeomorphism $h\co \hat{B}_1\ri\hat{B}_2$ such that $h({\bf x_1})={\bf x_2}$, $h(\mk p(z_{1;j}^\pm))=\mk p( z_{2;j}^\pm)$, and $\hat{u}_1=\hat{u}_2\circ h$. Suppose the boundary of $B_i$ consists of $p$ circles joined at nodes. Then, $B_i=(B_{i;1},\dots, B_{i;p})$, where each $B_{i;m}$ is a bubble domain with a single disk component,    $z_{i;0}\in B_{i;1}$, and $h(\prt B_{1;m})=\prt B_{2;m}$. The real locus of $\hat B_i$ is oriented by the induced orientation of $\prt B_i\subset B_i\subset \hat B_i$. Thus, $h$ maps $B_{1;1}$ to $B_{2;1}$ and  is orientation-preserving on $\prt B_{1;1}$. If $h$ is orientation-reversing on $\prt B_{1;m}$, we replace $u_2$ with  a different representative of $[u_2]$, namely 
  $$u'_2=(u_{2;1},\dots, u_{2;m-1}, \mk c (u_{2;m}),u_{2;m+1},\dots, u_{2;p}),
  $$
   with the corresponding changes of the marked points. As in the proof of Lemma~\ref{lem_dbl}, 
   $$\hat u_2=\hat u_2'\circ (\id,\dots,\id, \fc_{\hat{B}_{2;m}}\circ c_{\hat B_{2;m}},\id,\dots,\id)$$
   and thus $\hat u_1= \hat u'_2\circ h'$,
   where
   $$h'=(\id,\dots,\id, \fc_{\hat{B}_{2;m}}\circ c_{\hat B_{2;m}},\id,\dots,\id)\circ h 
   $$ is a real holomorphic diffeomorphism with  $h'(\prt B_{1;m})\ri\prt B_{2;m}$  orientation-preserving. Thus, we may assume that $h$ is orientation-preserving on the boundary, mapping \hbox{$B_1\subset \hat B_1$} to $B_2\subset \hat B_2$. Therefore, $h_{|B_1}$ relates $u_1$ and $u_2$ and thus 
    $$[u_1, {\bf x}_1, z_{1;0}^+, z_{1;1}^\pm,\dots z_{1;l}^\pm ]  =[u_2, {\bf x}_2, z_{2;0}^+, z_{2;1}^\pm,\dots z_{2;l}^\pm ] .$$
    This concludes the proof.
  \end{proof}
  
\begin{cor} \label{wsph} The map $\widetilde{\mathcal{M}}_{k,l+1}\rightarrow \overline{\mathbb{R}\mathfrak{M}}_{k,l+1}$ given by  
$$
 \big[C, {\bf x}, z_0^+, z_1^\pm\dots, z_l^\pm\big]\mapsto \big[\hat{C}, {\bf x}, \mk p(z_0^+),\mk p(z_1^\pm)),\dots ,  \mk p (z_l^\pm))\big]
 $$
 is  a homeomorphism.  
 \end{cor}

\begin{cor} There is a commutative diagram 
\[\begin{CD}
\M_{k,l+1}(M,A)
@>\ev_{x_i}\times \widetilde{\ev}_{z_j}\times \tau\circ \widetilde{\ev}_{z_j}>> L\times M\times M \\
@V\sD V\cong V @VV \id V\\
\overline{\R\mf}_{k,l+1}(M,A) @> \ev_{x_i}\times \ev_{z_j}\times \ev_{\bar{z}_j}>> L\times M\times M.
\end{CD}\] \end{cor}

\section{The relative sign of conjugation}\label{sec_rsc}

In order to study the orientability of the new moduli space $\M_{k,l+1}(M,A)$, we need   to know the (relative) sign of the isomorphism $g$ defined on the boundary of the moduli spaces $\widetilde{\mf}_{k,l+1}(M,b)$. Here we define and compute the \textsf{relative sign} of $g$. We use the canonical isomorphism 
\begin{equation}
    \label{eq_fpidm}
\Lt(\mf_{k_1,l_1+1}(M,b_1)\times_{fp}\mf_{k_2,l_2}(M,b_2))\cong\det(D)\otimes\Lt(\mf_{k_1,l_1+1}\times\mf_{k_2,l_2})
\end{equation}
 to separate the contributions from the determinant bundle of the linearization of $\bp$ and from the domain and compute the sign of each part in Lemmas \ref{lem_det} and   \ref{lem_dm}. We obtain the relative sign of $g$ 
 in Corollary \ref{cor_ms}.
   Lemma~\ref{lem_det} and Corollary \ref{cor_ms} are parallel to \cite[Theorem 1.3]{FOOO2} and \cite[Propositions 2.12, 5.1]{Sol} but we do not assume the existence of a relatively spin or pin structure on $L$. \\

 Given the bubble domain $B=D_1^2\coprod D_2^2/_{ \{1\}\sim \{-1\}}$ with a single boundary node,   a bundle pair $(E,F)\ri(B,\prt B)$, and a real Cauchy-Riemann operator $D^{(E,F)}$, there is a canonical isomorphism
\begin{equation}\label{eq_bub}
\det(D^{(E,F)})\cong \det(D_1^{(E,F)})\otimes \det(D_2^{(E,F)}) \otimes (\Lt F)^{\spcheck}_{|\text{node}} ,
\end{equation}
where $D_i^{(E,F)}$ is the operator on the $i$-th bubble component; see \cite[Section 2]{Zin}.  
Suppose $(E',F')\ri (B,\prt B)$ is another bundle pair such that
$$
(E',F')_{|D^2_1}=(E,F)_{|D^2_1}\qquad \text{and} \qquad (E',F')_{|D^2_2}\overset{\wt c}{\ri}(E,F)_{|D^2_2},
$$
where $\wt c$ is an anti-complex linear isomorphism covering $c_{D^2_2}$. For example,   given a map $u=(u_1,u_2)\co (B,\prt B)\ri (M,L)$, we can take
$$
(E,F)=(u_1,u_2)^*(TM,TL)\qquad\text{and}\qquad (E',F')=(u_1, \tau\circ u_2\circ c_{D^2_2})^*(TM,TL);
$$
the map $\wt c = u_2^*\text{d} \tau$ in this case.
The bundle isomorphism, to which we refer as \textsf{conjugation},
 \begin{gather*}\label{eq_cmap}
\begin{CD}
(E,F)    @>\id\coprod \wt c>>  (E',F')\\
@VVV        @VVV\\
(B,\prt B)     @>\id\coprod c_{D^2_2}>>  (B,\prt B)
\end{CD}
\end{gather*}
induces an isomorphism 
\begin{equation}\label{eq_det}
i_{(E,F)}:\det{D^{(E,F)}}\cong \det D^{( E', F')},
\end{equation}
where $D^{( E', F')}$ is the real Cauchy-Riemann operator on $(E',F')$ induced by~$D^{( E, F)}$. \\

 Given a complex bundle $(W, J_E)\ri (M, J)$ with a conjugation $\wt\tau$ covering $\tau$, the lines $\det \bar{\partial}^{(W,W^{\wt\tau})}$ form a line bundle over $\widetilde{\mf}_{k,l+1}^*(M,b)$; see \cite[Section~3.2]{EES} or \cite[Appendix~D.4]{Hua}. Using the map $i_{(W,W^{\wt\tau})}$ we then obtain a line bundle
 \begin{equation}\label{eq_glbdl2}
 \det\bp^{(W,W^{\wt\tau})}_{/i_{(W,W^{\wt\tau})}}\ri \M_{k,l+1}^*(M,A)
 \end{equation}
over the glued moduli space.\\

 By \cite[Proposition  4.9]{Geo1},  given a bundle pair $(E,F)\rightarrow (D^2, \partial D^2)$ and a real Cauchy-Riemann operator $D^{(E,F)}$, an orientation of $\det D^{(E,F)}$  is determined by a choice of trivializations of $F\oplus 3\det F$ over   $\partial D^2$ and of $\det F$ over  a point   $x_1\in \partial D^2$. Thus, a choice    of trivializations of $F\oplus 3\det F$ over   $\partial B$ and of $\det F$ over  the node    induces orientations on $\det(D_1^{(E,F)}), \det(D_2^{(E,F)})$, and $F^{\spcheck}_{|\text{node}}$  and hence on $\det(D^{(E,F)})$.
The push-forwards of these choices of trivializations by $\id\coprod c_{D^2_2}$ similarly induce an orientation on $\det D^{ (E', F')}$. \\

The \textsf{relative sign of $i_{(E,F)}$} is the sign of the isomorphism (\ref{eq_det}) when the two sides are oriented  using   trivializations of $F\oplus 3\det F$ over   $\partial B$ and of $\det F$ over  the node and their push-forwards by $\id\coprod c_{D^2_2}$.  By Lemma~\ref{lem_det}, this sign is independent of the choices of   trivializations and in fact depends only on the Maslov index $\mu(E_{| D^2_2},F_{|\prt D^2_2})$ of $(E_{| D^2_2},F_{|\prt D^2_2})$.
  Denote   by $\wt{w}_1(F_{|\prt D^2})$ the first Stiefel-Whitney class of $F$ evaluated on~$\prt D^2$ and considered as an integer, i.e. either 0 or 1.
 
\begin{lem} \label{lem_det}
The relative sign of $i_{(E,F)}$ equals $(-1)^\eps$, where
\[\eps = \frac{\mu(E_{| D^2_2},F_{|\prt D^2_2})+\wt{w}_1(F_{|\prt D^2_2})}{2}.\]
\end{lem}
\begin{proof}
Since the map $i_{(E,F)}$ is identity on the first bubble, the map factors through the isomorphism (\ref{eq_bub}) with the only nontrivial factor being the middle one. Thus, the relative sign of $i_{(E,F)}$ is the sign of the conjugation on the second bubble.
Let
$$
 (E^1,F^1)=({\det}_{\mathbb{C}}E_{| D^2_2},  \det F_{|\prt D^2_2})\quad\text{and}\quad(\wt{E},\wt{F})=(E_{| D^2_2} \oplus 3E^1, F_{|\prt D^2_2}\oplus 3 F^1).
$$
The isomorphism $\wt c$ induces anti-complex linear isomorphisms
$$
\wt c:E^1\ri E'^1, \quad \wt c: 4E^1\ri 4E'^1, \quad \wt c:\wt E\ri \wt E'.
$$
We induce an orientation on $\det(D_2^{(E,F)})$ by orienting $\det(\bp^{(\wt{E},\wt{F})})$, $\det(\bp^{(E^1,F^1)})$, and $\det(\bp^{(4E^1,4F^1)})$ using the isomorphism 
\begin{equation}\label{eq_t11}
\det (D_2^{(E,F)})\otimes \det(\bp^{(4E^1,4F^1)})\cong \det(\bp^{(\wt{E},\wt{F})})\otimes\det(\bp^{(E^1,F^1)});
\end{equation}
see the proof of \cite[Theorem 1.1]{Geo1}.
  We determine the sign of the isomorphism on $\det (D_2^{(E,F)})$ by determining it on  each of the remaining factors in~(\ref{eq_t11}). \\
 
 We orient $\det(\bp^{(\wt{E},\wt{F})})$ and $\det(\bp^{(\wt{E}',\wt{F}')})$ by trivializing 
 $$
 (\wt{E}_{|U},\wt{F})\overset{\phi}{\cong} (\C^{n+3}\times U,\R^{n+3}\times \prt D^2_2)\quad\text{and}\quad  (\wt{E}_{|U}',\wt{F}')\overset{\psi}{\cong} (\C^{n+3}\times U,\R^{n+3}\times \prt D^2_2)
 $$
  in a neighbourhood $U$ of $\prt D^2_2$,  with $\psi\circ \wt c\circ\phi^{-1}=(c_{\C^{n+3}}, c_{D^2_2})$, via the given trivialization of~$\wt{F}$  and pinching the disk $D^2_2$ to obtain a sphere attached to a disk at a point~$p$. 
  We can pinch along a circle which is invariant under $c_{D^2_2}$ so that $c_{D^2_2}$ induces the standard anti-holomorphic involutions on the pinched off disk and sphere. 
  Thus, the conjugation map factors through the isomorphism 
  $$
  \det(\bp^{(\wt{E},\wt{F})})\cong \det (\bp_{D^2})\otimes \det{\bp_{S^2}}\otimes (\Lt_\R \wt E)_{|p}^{\spcheck}.
  $$ 
  The indices over the disks are isomorphic to  $\ind(\bp^{(\mathbb{C}^{n+3},\mathbb{R}^{n+3})})\cong \mathbb{R}^{n+3}$ by evaluation at the node and $c_{\C^{n+3}}$ acts as identity on it. The orientation on the determinant line on the sphere is induced from the canonical complex orientations on the kernel and cokernel. 
  The complex dimension of the index on the sphere is $c_1(\wt{E})\cdot S^2+n+3$.
     Since $\wt c$ is anti-complex linear, the sign of the conjugation map on the sphere component is $(-1)^{c_1(\wt {E})\cdot S^2+n+3}$. Since    
     $$
     c_1(\wt{E})\cdot S^2=\frac{\mu(\wt{E},\wt{F})}{2}=2\mu(E_{|D^2_2},F_{|\prt D^2_2}),$$ 
 there is no contribution to the sign   coming from the Chern class. 
 Finally, $\wt c$ acts on the canonical orientation of the  incident condition $\wt{E}_p^{\spcheck}$ with a sign   $(-1)^{n+3}$, since the complex dimension of $\wt E$ is $n+3$. Combining the three contributions gives zero and thus the conjugation map is orientation-preserving on the $\det(\bp^{(\wt{E},\wt{F})})$ factor. Similarly, the map is orientation-preserving on the $\det(\bp^{(4E^1,4F^1)})$ factor, where we use the canonical spin structure on $4F^1$ to orient. \\
 
  Similarly, we orient $\det(\bp^{(E^1,F^1)})$ by  choosing an isomorphism 
  $$
 ({E^1}_{|U},{F^1})\overset{\phi}{\cong} (\C\times U, e^{\frac{i\theta \wt{w}_1(F^1)}{2}}\R\times \prt D^2_2) 
 $$
  in a neighbourhood $U$ of $\prt D^2_2$,  with $\psi\circ \wt c\circ\phi^{-1}=(c_{\C}, c_{D^2_2})$ using the given trivialization of~$F^1_{|\text{node}}$  and pinching the disk $D^2_2$ to obtain a sphere attached to a disk at a point~$p$. The Maslov index on the pinched off disk is then equal to $\wt{w}_1(F^1)$ and the Chern class on the sphere is equal to $\tfrac{1}{2}(\mu(E^1,F^1)-\wt{w}_1(F^1))$; see the proof of \cite[Proposition 3.2]{Geo1}.
     The determinant line of the operator on the disk  is isomorphic to $ \det(F^1)_{|\text{node}}^{\wt{w}_1(F^1)+1}$. This isomorphism depends on the orientation of $\partial D^2$ if and only if   $\wt w_1(F^1) =1$. We use this orientation  to determine the direction along which we transport $F^1_{|x_2}$ to $F^1_{|\text{node}}$;   reversing the orientation of~$\partial D^2$  reverses the orientation of the determinant line. Since  $c_{D^2_2}$ reverses the orientation of~$\partial D^2_2$ and $\wt c$  preserves the orientation of $\det(F^1)_{|\text{node}}$, the  contribution to the sign  is $(-1)^{\wt{w}_1(F^1)}$. As above,  the contribution to the sign of $i_{(E,F)}$ from the incident condition is $(-1)$ and from  the sphere  is  $(-1)$ to the power of the index, which is 
  $$c_1(\det E^1)
\cdot(S^2)+1=\frac{\mu(E^1,F^1)-\wt{w}_1(F^1)}{2}+1.$$
 Thus, the overall sign of the conjugation map is 
 $$(-1)^{\wt{w}_1(F^1)+1+\frac{\mu(E^1,F1)-\wt{w}_1(F^1)}{2}+1}=(-1)^{\frac{\mu(E^1,F^1)+\wt{w}_1(F^1)}{2}},$$
 as claimed.
\end{proof}

\begin{remark} \label{rem_indor}For the applications in the next section, it is useful to compute the relative sign of $i_{(E,F)}$ given trivializations of $F\oplus 3\det{F}$ over $\prt B$ and a trivialization of $\det F$ over some point $x_1\in\prt B$ different from the node. We can still orient the two sides of~(\ref{eq_det}): the left side is oriented by transporting the trivialization of $\det F_{|x_1}$ to $\det F_{|\text{node}}$ using the positive direction of $\prt B$ and the right side is oriented by pushing forward the given trivializations and transporting $\det F'_{|\id\coprod c_{D_2^2}(x_1)}$ to $\det F'_{|\text{node}}$ using the positive direction of~$\prt B$.     This does not change  the sign of $i_{(E,F)}$ if  $F$ is orientable, but  there is an additional contribution, described in the following corollary, if $F$ is not orientable.
 \end{remark}

 \begin{cor}\label{cor_det} If the two sides of (\ref{eq_det}) are oriented as in Remark \ref{rem_indor} using trivializations of $F\oplus 3\det{F}$ over $\prt B$ and   $\det F_{|x_1}$, $x_1\neq$ node, then the sign of $i_{(E,F)}$ is $(-1)^\eps$, where
 \[
 \eps  = \frac{\mu((E_{| D^2_2},F_{|\prt D^2_2}))+\wt{w}_1(F_{|\prt D^2_2})}{2} + \delta_{x_1}\wt w_1(F_{|\prt D^2_1})\wt w_1(F_{|\prt D^2_2}),
 \] with $\delta_{x_1}=1$ if $x_1\in \prt D^2_2$ and zero otherwise.
 \end{cor}
 
 \begin{proof} We orient the left side of (\ref{eq_det}) by transporting the   trivialization of $\det F_{|x_1}$ to $\det F_{|\text{node}}$. If we orient the right side with these choices of trivializations over $\prt B$ and the node, then the sign is given by Lemma \ref{lem_det}. However, the right hand side is oriented by transporting $\det( F')_{|\id\coprod c_{D^2_2}(x_1)}$ to the node using the positive orientation of $\prt B$.
   We consider two cases:
\begin{itemize}
\item $x_1$ is on the first bubble, 
\item $x_1$ is on the second bubble.
\end{itemize}
\noindent
In the first case, the map $i_{(E,F)}$ is identity on $\prt D^2_1$, and this change does not affect  the trivialization of  $\det F'_{|\text{node}}$. Hence, the right side is oriented   as in Lemma \ref{lem_det}   and  there is no additional contribution to the  sign of $i_{(E,F)}$.

In the second case, the trivialization of $\det  F'_{|\id\coprod c_{D^2_2}(x_1)}$ is transported to the node using the positive direction of $\prt D^2_2$. The difference with the   push-forward trivialization on the node induced by  transporting $\det F_{|x_1}$ is measured by $\wt w_1(F_{|\partial D^2_2})$, since the latter is equivalent to transporting $\det  F'_{|\id\coprod c_{D^2_2}(x_1)}$ using the negative direction of $\prt D^2_2$.  The two trivializations are the same if $\wt w_1(F_{|\partial D^2_2})=0$, in which case there is no additional contribution to the sign of $i_{(E,F)}$. If $\wt w_1(F_{|\partial D^2_2})=1$, the induced trivialization on the node changes.   
This, furthermore,   changes the orientation of $\det D_1^{(E',F')}$ if \hbox{$\wt w_1(F_{|\partial D^2_1})=0$}, since in this case the determinant line of $\bp_1^{(E'^1,F'^1)}$ is isomorphic to a single copy of $F'^1_{|\text{node}}=\det F'_{|\text{node}}$; there is no change in the orientation of $\bp_1^{(E'^1,F'^1)}$  if $\wt w_1(F_{|\partial D^2_1})=1$.
Since we are considering the case when $\wt w_1(F_{|\partial D^2_2})=1$, there is no change in the orientation of $\bp_2^{(E'^1,F'^1)}$ and hence of $\det D_2^{(E',F')}$.  
Thus, the additional contribution in this case comes from the change of the orientation of the node and the change of the orientation of $\det D_1^{(E',F')}$, which can be expressed as 
$$(-1)^{1+(\wt{w}_1(F_{|\prt D^2_1})+1)}=(-1)^{\wt{w}_1(F_{|\prt D^2_1})}.
$$
Thus, the correction to the sign in Lemma \ref{lem_det} is 
$ \delta_{x_1} \wt{w}_1(F_{|\partial D^2_1})\wt{w}_1(F_{|\partial D^2_2}).$
\end{proof}

We now turn to the sign of $g$ when the target manifold is a point, which describes the moduli space of domains.
 The moduli space $\overline{\mf}_{k,l}$ is orientable. As in the proof of \cite[Corollary 1.8]{Geo1},  we orient it by assigning $+$ to the point $\overline{\mf}_{1,1}$ and use the forgetful maps $\mf_{k,l}\ri \mf_{1,1}$ to induce orientation on $\mf_{k,l}$ using the canonical orientation on the fiber, which is an open subset of $ (\prt D^2)^{k-1}\times (D^2)^{l-1}$. 
 When we work with the decorated moduli space $\wt{\mf}_{k,l}$, we slightly change these choices of orientations: if an interior marked point has a decoration $+$, we induce the orientation on the fiber using the canonical orientation of $D^2$; if the decoration is $-$, we take the opposite orientation of $D^2$. Thus, the point $\mf_{1,1^+}$ is assigned the sign $+$ and  the point $\mf_{1,1^-}$ is assigned the sign~$-.$ The map $g$ is defined on the boundary strata 
 $$\wt{\mf}_{k_1,l_1+1}\times \wt{\mf}_{k_2,l_2}\subset \wt{\mk M}_{k,l+1},$$ where the last boundary marked point on the first bubble is identified with the first boundary marked point on the second bubble. 
 
\begin{lem}\label{lem_dm}
The sign of $g$ on $\wt{\mf}_{k_1,l_1+1}\times \wt{\mf}_{k_2,l_2}$ equals $ (-1)^{k_2}.$
\end{lem}
\begin{proof}
 The map $g$ is defined as identity on the first factor and as conjugation, reversing the decorations, on the second.  Thus, the contribution to the sign comes from the conjugation on the second factor. The conjugation is the identity on $\mf_{1,1}$;   since $\mf_{1,1^+}$ and $\mf_{1,1^-}$ are oriented differently, the conjugation is orientation-reversing as a map $\mf_{1,1^+}\ri\mf_{1,1^-}$. Thus,  $g$ is orientation-reversing if $(k_2,l_2)=(1,1)$. If $l_2>0$, the conjugation on $\wt{\mf}_{k_2,l_2}$ factors through the forgetful map $\wt{\mf}_{k_2,l_2}\ri\wt{\mf}_{1,1^\pm}$; it remains   to compute the contribution from the fiber, which is an open subset of \hbox{$(\prt D^2)^{k_2-1}\times (D^2)^{l_2-1}$.} Each boundary marked point contributes $(-1)$ to the sign, whereas the map on the interior marked points is orientation-preserving, since we take the opposite orientation when the decoration is $-$. Thus, the   contribution from the base and the fiber is $1+k_2-1$. If $l_2=0$, the orientation on $\wt{\mk{M}}_{k_2,0}=\ov{\mk{M}}_{k_2,0}$ is induced by the two fibrations $\ov{\mk M}_{k_2,1}\ri\ov{\mk M}_{k_2,0}$ and $\ov{\mk M}_{k_2,1}\ri \ov{\mk M}_{1,1}$. The sign of $g$ on $\ov{\mk M}_{1,1}$ is +;   the contribution from the fibers is $(-1)$ for the first fibration and $(-1)^{k_2-1}$ for the second. Thus, the overall sign is again   $(-1)^{k_2}$.
 \end{proof}

\begin{proof}[{\bf\emph{Proof of Proposition \ref{cor_dmor}}}]
The spaces 
$\M_{1,1}$ and $\M_{2,1}$ are a point  and a circle, respectively, and the latter is canonically oriented by the orientation of $\wt{\mk M}_{2,1}$. The map $g$ on $\wt{\mf}_{1,l_1+1}\times \wt{\mf}_{1,l_2}\subset \wt{\mf}_{0,l+1}$ is orientation-reversing by Lemma \ref{lem_dm}.   Thus, $\M_{0,l+1}$ is orientable and canonically oriented by the canonical orientation of $\wt{\mk M}_{0,l+1}$. \\

We show that all other cases are non-orientable by writing explicit loops on which $w_1(\M_{k,l+1})$ is nonzero. We either have at least 1 interior and 3 boundary marked points or at least 2 interior and 1 boundary marked points. In the former case, there is a loop in the moduli space crossing the following boundary strata: the first stratum has   only the first 3 boundary marked points $(x_1,x_2,x_3)$ on a separate bubble, the next has  only $(x_2,x_1)$ on a separate bubble,   the next only $(x_3,x_1)$, and the last only $(x_3,x_2)$. Crossing these boundaries closes a loop. The sign of $g$ on the first boundary is orientation-preserving, whereas it is orientation-reversing on the other 3. Thus, $w_1(\M_{k,l+1})$ does not vanish on this loop. If there are at least 2 interior and 1 boundary marked points, we take a loop crossing two boundary strata: one having the interior marked point~$z_1$ and a boundary marked point on a separate bubble and one having only $z_1$ on a separate bubble. The map $g$ is orientation-preserving on the former boundary and orientation-reversing on the latter. Thus, $w_1(\M_{k,l+1})$ does not vanish on this     loop. By this argument, the first Stiefel-Whitney class $w_1(\M_{k,l+1})$ is supported on the divisors $D^{\text{odd}}$ having an odd number of boundary marked points on the second bubble (in addition to the node).
\end{proof}

\begin{remark} The real moduli space $\ov{\R\mk M}_{k,0}$ is not orientable if $k>4$ as shown in \cite{Cey} and \cite{Eti}.  
\end{remark}
 
We   now compute the relative sign of the involution  
$$
g:  \wt{\mf}_{k_1,l_1+1}(M,b_1)\times_{fp}\wt{\mf}_{k_2,l_2}(M,b_2)\ri \wt{\mf}_{k_1,l_1+1}(M,b_1)\times_{fp}\wt{\mf}_{k_2,l_2}(M,\bar{b}_2).
$$
We induce an orientation at a point 
 $$
 u \in \wt{\mf}_{k_1,l_1+1}(M,b_1)\times_{fp}\wt{\mf}_{k_2,l_2}(M,b_2)
 $$ 
      via choices of trivializations as in Remark \ref{rem_indor} with $F=TL$. The chosen trivializations induce an orientation on $\det D_u^{(TM,TL)}$ and this together with the canonical orientation of the moduli space of domains induces an orientation at $u$ via (\ref{eq_fpidm}).  If the second bubble  is not stable, we stabilize locally by adding an interior marked point. The differential of $g$ factors through the isomorphism (\ref{eq_fpidm}), inducing the map (\ref{eq_det}) on $\det D_u^{(TM,TL)}$. The above choices of trivializations similarly induce an orientation at the point $g(u)$. We  call the   sign of $g$ with respect to these orientations the \textsf{relative sign of~$g$}. 
  Thus, the relative sign of~$g$ at $u$ is the sum of the contributions from the determinant bundles,  as given by Corollary \ref{cor_det} with $(E,F)=(TM, TL)$, and the moduli space of domains, as given by  Lemma \ref{lem_dm}. If the second bubble  is not stable, the  interior marked point added to stabilize the domain is not decorated and thus it contributes to the sign on the moduli space of domains. This addition is canceled by the additional contribution coming from the fiber of the map forgetting the additional marked point.  Thus, we obtain the following corollaries.

\begin{cor}\label{cor_ms}
The relative sign of $g$ defined above equals $(-1)^\eps$, where
\begin{equation}\label{eq_ms}
\eps = \frac{\mu(TM_{|D^2_2},TL_{|\prt D^2_2})+\wt{w}_1(TL_{|\partial D^2_2})}{2}+\delta_{x_1} \wt{w}_1(TL_{|\partial D^2_1})\wt{w}_1(TL_{|\partial D^2_2}) + k_2.
\end{equation}
\end{cor}

 \section{Orientability of $\widetilde{\mathcal{M}}_{k,l+1}^*(M,A)$}\label{sec:or}

In this section, we compute the first Stiefel-Whitney class of the new moduli space $\M_{k,l+1}^*(M,A)$. 
If this class is a pull-back of a class $\kappa$ on the product $L\times \M_{k,l+1}$ and $M$ is strongly semi-positive, the image of $\M_{k,l+1}(M, A)$ under $\ev\times \mk f$ is a pseudocycle with coefficients in the local system twisted by $\kappa$. In Theorem \ref{gc}, we express $w_1(\widetilde{\mathcal{M}}_{k,l+1}^*(M,A))$ in terms of classes from $L$, $\M_{k,l+1}$, and Poincare duals of boundary divisors in $\M_{k,l+1}^*(M,A)$. The presence of the latter shows that in general the target does not carry the correct local system needed to accommodate the pseudocycle. In Theorem~\ref{cor_rs}, 
we give sufficient conditions for the class to vanish if $k=0$ and to be as close as possible to a pull-back if $k>0$.  In the latter case, one can still  define invariant counts  using intersection theory arguments, as explained in the next section.\\

   Denote by  
\begin{enumerate}[label=$\bullet$, leftmargin=*]
\item $D_i$   the codimension one strata of $\widetilde{\mathcal{M}}_{k,l+1}^*(M,A)$  
with Maslov index on the second bubble satisfying  $\mu(TM_{|D^2_2}, TL_{|\prt D^2_2})\equiv i$ mod $4$, 
\item  $O_{x_1}$   the codimension one strata of $\widetilde{\mathcal{M}}_{k,l+1}^*(M,A)$ having the marked point $x_1$ on the second bubble and odd Maslov indices on both bubbles,
\item $U$   the codimension one strata of $\widetilde{\mathcal{M}}_{k,l+1}^*(M,A)$ having a single boundary marked point on the second bubble (in addition to the node).
\end{enumerate}
Each of these strata has no boundary in $\widetilde{\mathcal{M}}_{k,l+1}^*(M,A)$ and thus defines a class in $H_{\text{top}-1}^{\text{BM}}(\widetilde{\mathcal{M}}_{k,l+1}^*(M,A);\Z_2)$, where $H^{\text{BM}}_*$ denotes Borel-Moore homology; see \cite[Chapter IX]{Iver}. Denote by $D_i^{\spcheck}, O_{x_1}^{\spcheck},$ and $U^{\spcheck}$ their Poincare dual classes in $H^1(\widetilde{\mathcal{M}}_{k,l+1}^*(M,A);\Z_2)$.

\begin{thm}\label{gc}  Suppose $M$ is a symplectic manifold with an anti-symplectic involution~$\tau$ and $k,l\in \Z^{\geq 0}$ with $k+l>0$; if $L=M^\tau$ is not orientable, assume $k>0$.   Let  $\gamma$ be a loop in $\M_{k,l+1}^*(M,A)$ and    $\widetilde{T}_\gamma$ be the   class in $H_2(L;\mathbb{Z}_2)$ traced in $L$ by the boundary  of the domains  along  $\gamma$.    The first Stiefel-Whitney class of the moduli space $\M_{k,l+1}^*(M,A)$ evaluated on    $\gamma$ equals 
 \begin{equation}\label{eq_fswc} \begin{split}w_1(\widetilde{\mathcal{M}}_{k,l+1}^*(M,A))\cdot \gamma = &\, w_2(L)\cdot \widetilde{T}_{\gamma}+ (c_1(M)\cdot A+1)\cdot(w_1(L)\cdot \ev_{x_1}(\gamma))\\
&+\left[\mathfrak{f}^*(w_1(\widetilde{\mathcal{M}}_{k,l+1}))+D^{\spcheck}_1+D_2^{\spcheck} +  O_{x_1}^{\spcheck}+U^{\spcheck}\right]\cdot\gamma.
\end{split}
\end{equation}
 In particular, if $L$ is orientable, then  
$$ w_1(\widetilde{\mathcal{M}}^*_{k,l+1}(M,A))\cdot \gamma=w_2(L)\cdot \widetilde{T}_{\gamma}+  \left[\mathfrak{f}^*(w_1(\widetilde{\mathcal{M}}_{k,l+1}))+ D_2^{\spcheck}  +U^{\spcheck}\right]\cdot\gamma.
$$
  \end{thm}

 \begin{proof} 
 Let $k>0$. We may assume that the loop $\gamma$ consists of   several paths $\alpha_1, ..\, , \alpha_r$ each in some moduli space $\widetilde{\mf}_{k,l+1}(M, b_i)$, with end points at  boundary divisors, glued together by the map $g$. Choose a point  $u$ in the interior of the path $\alpha_1$; this   separates the path $\alpha_1$ into two paths $\beta_1$ and $\beta_2$. Choose trivializations of $TL\oplus 3\det TL $ and $\det TL$ over the images $u(\partial D^2)$ and $u(x_1)$, respectively. The path $\beta_1$ transports these trivializations to its other end, and hence to the beginning of $\alpha_2$, which   transports them to $\alpha_3$ and so on until we close the loop.  The first Stiefel-Whitney class evaluated on the loop $\gamma$ is then given by the sum mod 2 of
 
\begin{enumerate}[label=$\bullet$, leftmargin=*]
 \item the relative signs of the map $g$ at  the boundary divisors we cross plus the number of the crossed divisors (if $g$ is orientation-reversing on all boundary divisors, there is no contribution), 
 \item   the difference between the initial trivialization of $TL\oplus3\det TL$ at $u_{|\prt D^2}$ and  the one induced following the loop,
 \item  the difference between the initial trivialization of $\det TL_{|u(x_1)}$  and  the one induced following the loop if $\wt w_1(TL_{|u(\partial D^2)}) =0$.
 \end{enumerate} 
 \noindent
 By Lemma \ref{lem_sgn}, the first set of signs describes the change of orientation across codimension one boundary.
 The last two sets of signs describe the change of the orientation of the moduli space at~$u$ under changes of the given trivializations; see \cite[Proposition~4.9]{Geo1}. \\
 
 Since $TL\oplus 3 \det TL$ is an orientable bundle, the change in the trivialization is measured by 
 $$w_2(TL\oplus 3\det TL)\cdot\wt{T}_{\gamma} =w_2(TL)\cdot\wt{T}_{\gamma}.$$
This gives  the first summand in (\ref{eq_fswc}).  The change in the trivialization of $\det TL_{|u(x_1)}$ is measured by $w_1(\det TL)\cdot \ev_{x_1}(\gamma)$; by \cite[Proposition~4.9]{Geo1}, its contribution to the change of the orientation of the moduli space is
 $$
 (w_1(\det TL_{|u(\prt D^2)})+1)(w_1(\det TL)\cdot \ev_{x_1}(\gamma)).
 $$
Since $w_1(TL)\cdot (u(\prt D^2)=\mu(b)$ mod 2 and $\mu(b) = c_1(M)\cdot A$, this gives the second summand in~(\ref{eq_fswc}).\\

Finally, the relative sign of $g$ at the stratum $\wt{\mf}_{k_1,l_1+1}(M,b_1)\times_{fp}\wt{\mf}_{k_2,l_2}(M,b_2, )$ is given by Corollary~\ref{cor_ms}.
The contribution from the marked points is the same as their contribution in the moduli space of domains if   the  domain is     stable. The unstable domains have a second bubble with either no marked points or a single boundary marked point. The former contributes one for the node and one for the boundary component and   can be disregarded.  The contribution from the latter is non-trivial and is described by $U^{\spcheck}\cdot\gamma$. This gives the expression $[\mathfrak{f}^*(w_1(\widetilde{\mathcal{M}}_{k,l+1}))+U^{\spcheck}]\cdot\gamma$ in~(\ref{eq_fswc}).
The fraction   in~(\ref{eq_ms}) is nonzero mod $2$ if $\mu(TM_{|D^2_2},TL_{|\prt D^2_2})$ is   1 or 2 mod 4. This  is described by $D_1^{\spcheck}\cdot\gamma$ and $D_2^{\spcheck}\cdot\gamma$. The expression  $\displaystyle \delta_{x_1} \wt{w}_1(TL_{|\partial D^2_1})\wt{w}_1(TL_{|\partial D^2_2})$ contributes   when the boundary marked point $x_1$ is on the second bubble and the Maslov indices of both bubbles are odd, which  is described by $
O_{x_1}^{\spcheck}\cdot\gamma$. This gives~(\ref{eq_fswc}).\\

 If $L$ is orientable, we do not need a marked point to transport the chosen trivialization  of $\det TL$ over a point on $u(\prt D^2)$, as any such choice determines an orientation on $L$ and  a continuous choice of trivializations of $\det TL$. Thus, the second term, measuring the difference in this trivialization, vanishes. Moreover, the orientability of $L$ implies that all Maslov indices are even and hence there is no contribution from $D_1$ or~$O_{x_1}$, as~well. \end{proof}

\begin{remark} If the loop $\gamma$ is inside one of the moduli spaces $\widetilde{\mf}_{k,l+1}(M,b)$,       the formula of Theorem \ref{gc} reduces to the formula for  $\widetilde{\mf}_{k,l+1}(M,b)$ given in \cite[Theorem 1.1]{Geo1}. Results on the orientability of the precompactified real moduli space are also obtained   in~\cite{Rem}. Other expressions for the first Stiefel-Whitney class of $\ov{\R{\mk M}}_{k,l}(M, A)$, when   $\dim M \leq 3$ and the number of marked points equals the dimension of the moduli space, are given in \cite[Proposition 4.5]{Wel} and \cite{Pi, Pi2}. 
\end{remark}

Similarly, if $(E, J_E)\ri (M, J)$ is a complex bundle  with a conjugation~$\wt\tau$ lifting~$\tau$, $F=E^{\wt\tau}$, we compute the first Stiefel-Whitney class of the line bundle $\det \bp^{(E,F)}_{/i_{(E,F)}}$; see (\ref{eq_glbdl2}) for the definition.
Let
\begin{enumerate}[label=$\bullet$, leftmargin=*]
\item $d_i^{(E,F)}$   the codimension one strata of $\widetilde{\mathcal{M}}_{k,l+1}^*(M,A)$  
with Maslov index on the second bubble satisfying  $\mu(E_{|D^2_2}, F_{|\prt D^2_2})\equiv i$ mod $4$, 
\item  $o^{(E,F)}_{x_1}$   the codimension one strata of $\widetilde{\mathcal{M}}_{k,l+1}^*(M,A)$ having the marked point $x_1$ on the second bubble and odd Maslov indices $\mu(E_{|D^2_i}, F_{|\prt D^2_i})$ on both bubbles.
\end{enumerate}

\begin{prop}\label{gcdet}
Suppose $M, \tau$, and  $L$ are as in Theorem~\ref{gc}, $k,l\in \Z^{\geq 0}$ with $k+l>0$,  and   $(E, J_E)\ri (M, J)$ is a complex bundle  with a conjugation~$\wt\tau$ lifting~$\tau$, \hbox{$F=E^{\wt\tau}$};
if $F$ is not orientable, assume $k>0$.   Let  $\gamma$ be a loop in $\M_{k,l+1}^*(M,A)$ and    $\widetilde{T}_\gamma$ be the   class in $H_2(L;\mathbb{Z}_2)$ traced in $L$ by the boundary  of the domains  along ~$\gamma$.    The first Stiefel-Whitney class of the bundle $\det \bp^{(E,F)}_{/i_{(E,F)}}$ as in (\ref{eq_glbdl2}) evaluated on $\gamma$ equals
 \begin{equation*}\begin{split} w_1(\det \bp^{(E,F)}_{/i_{(E,F)}})\cdot \gamma =  w_2( F)\cdot \widetilde{T}_{\gamma}+ (c_1(E)\cdot A+1)\cdot(w_1(F)\cdot \ev_{x_1}(\gamma))&\\
+\left[d^{\spcheck(E,F)} _1+d_2^{\spcheck(E,F)} +  o^{\spcheck(E,F)}_{x_1}\right]\cdot\gamma&.
\end{split}
\end{equation*}
  In particular, if $F$ is orientable, then  
$$ w_1(\det \bp^{(E,F)}_{/i_{(E,F)}})\cdot \gamma=w_2(F)\cdot \widetilde{T}_{\gamma}+   d_2^{\spcheck(E,F)}   \cdot\gamma.$$
\end{prop}

\begin{proof}
The fiber of $\det \bp^{(E,F)}_{/i_{(E,F)}}$ at a point is oriented by trivializations of $ F\oplus 3\det F$ and $\det F$ over the image of the boundary of the disk and a point on this image, respectively. As in the proof of Theorem \ref{gc}, the contribution to the first Stiefel-Whitney class is given by the relative signs of the gluing map given in Corollary \ref{cor_det} and the first and second Stifel-Whitney classes of $F$, giving the expressions in the statement.
\end{proof}

If the pair $(E,F)$ is as in Proposition \ref{gcdet}, the isomorphism
$$
\det \bp^{(E\oplus TM, F\oplus TL)}_{/i_{(E\oplus TM, F\oplus TL)}}\cong \det \bp^{(E,F)}_{/i_{(E,F)}}\otimes \det \bp^{(TM,TL)}_{/i_{(TM,TL)}}
$$
induces the equality
\begin{equation}\label{eq_fsw2}
w_1(\det \bp^{(E\oplus TM, F\oplus TL)}_{/i_{(E\oplus TM, F\oplus TL)}})= w_1(\det \bp^{(E,F)}_{/i_{(E,F)}})+ w_1(\det \bp^{(TM,TL)}_{/i_{(TM,TL)}}).
\end{equation}

\begin{cor}\label{prop_cp}
Suppose $M$ is a symplectic manifold with an anti-symplectic involution~$\tau$ and $k,l\in \Z^{\geq 0}$ with $k+l>0$; if $L=M^\tau$ is not orientable, assume $k>0$. If     $E\ri M$ is a complex bundle   with a conjugation $\wt\tau$ lifting~$\tau$, $F=E^{\wt\tau}$, and
$$w_1(F)=w_1(L),~~ w_2 (F\oplus TL)=0,~~ 4\,|\,(c_1(E)+c_1(M))\cdot A\quad \forall A\in \H_2(M),$$
then 
$$w_1(\widetilde{\mathcal{M}}_{k,l+1}^*(M,A))=w_1(\det \bp^{(E,F)}_{/i_{(E,F)}})+\mathfrak{f}^*(w_1(\widetilde{\mathcal{M}}_{k,l+1}))+ U^{\spcheck} .$$
In particular, if $k=0$, then 
$$w_1(\widetilde{\mathcal{M}}_{0,l+1}^*(M,A))=
 w_1(\det \bp^{(E,F)}_{/i_{(E,F)}}). $$ 
\end{cor}

\begin{proof}
By our assumption, Proposition \ref{gcdet}, and (\ref{eq_fsw2}),
$$
0= w_1(\det \bp^{(E,F)}_{/i_{(E,F)}})+ w_1(\det \bp^{(TM,TL)}_{/i_{(TM,TL)}}).
$$
By Proposition \ref{gcdet} and Theorem \ref{gc},
$$w_1(\widetilde{\mathcal{M}}_{k,l+1}^*(M,A))=w_1(\det \bp^{(TM,TL)}_{/i_{(TM,TL)}})+\mathfrak{f}^*(w_1(\widetilde{\mathcal{M}}_{k,l+1}))+ U^{\spcheck},$$
giving the result.
\end{proof}

\begin{thm}\label{cor_rs} Suppose $M$ is a symplectic manifold with an anti-symplectic involution~$\tau$ and $k,l\in \Z^{\geq 0}$ with $k+l>0$. If    $L=M^\tau$ is orientable and there is a complex bundle $E\ri M$,   with a conjugation $\wt\tau$ lifting~$\tau$, such that 
$$w_2 (2E^{\wt\tau}\oplus TL)=0, \quad4\,|\,(\mu(2E,2E^{\wt\tau})+\mu(M,L))\cdot b\quad\forall b\in \H_2(M,L)~\text{with}~b=\bar b, $$
then 
$$w_1(\widetilde{\mathcal{M}}_{k,l+1}^*(M,A))= \mathfrak{f}^*(w_1(\widetilde{\mathcal{M}}_{k,l+1}))+ U^{\spcheck}.
$$
In particular,
$\widetilde{\mathcal{M}}_{0,l+1}^*(M,A)
$ is orientable.
\end{thm}

\begin{proof} 
  If all Maslov indices 
  $$\mu(b)=(\mu(2E,2E^{\wt\tau})+\mu(M,L))\cdot b$$
    are divisible by  $4$,  
  \begin{equation}\label{eq_some} w_1(\widetilde{\mathcal{M}}_{k,l+1}^*(M,A))=w_1(\det \bp^{(2E,2E^{\wt\tau})}_{/i_{(2E,2E^{\wt\tau})}})+\mathfrak{f}^*(w_1(\widetilde{\mathcal{M}}_{k,l+1}))+ U^{\spcheck}
  \end{equation}
  by Corollary \ref{prop_cp} .
If not all Maslov indices are divisible by $4$, but this holds    for $b$ such that $b= \bar{b}\in \H_2(M,L)$, we  show that the contribution from $d_2^{\spcheck(2E\oplus TM, 2E^{\wt\tau}\oplus TL)}\cdot \gamma$ vanishes for all loops $\gamma$ in the moduli space $\M^{*}_{k,l+1}(M,A)$ and (\ref{eq_some}) still holds. The overall contribution from the Maslov indices to the sum expressing $w_1(\det \bp^{(2E\oplus TM, 2E^{\wt\tau}\oplus TL)}_{/i_{(2E\oplus TM, 2E^{\wt\tau}\oplus TL)}})\cdot \gamma$ is     $\sum_{i=1}^s \frac{\mu(b_{i,2})}{2}$ mod $2$; see the proof of Theorem \ref{gc}.   Since $\gamma$ is a loop, the differences  $(b_{i,2}-\bar{b}_{i,2})$ must sum to zero.   This implies 
$$\sum_{i=1}^s b_{i,2}=\sum^s_{i=1}\bar b_{i,2}=\sum_{i=1}^s -\tau_* b_{i,2}=-\tau_* \sum_{i=1}^s b_{i,2}=\ov{\sum_{i=1}^s b_{i,2}}.$$
Therefore, $\mu(\sum_{i=1}^s b_{i,2})\equiv 0$ mod 4, and thus $d_2^{\spcheck(2E\oplus TM, 2E^{\wt\tau}\oplus TL)}\cdot \gamma=0$.
Thus, by Proposition~\ref{gcdet} and (\ref{eq_fsw2}),
$$0=w_1(\det \bp^{(2E,2E^{\wt\tau})}_{/i_{(2E,2E^{\wt\tau})}})+ w_1(\det \bp^{(TM,TL)}_{/i_{(TM,TL)}})$$
 and    as in the proof of Corollary \ref{prop_cp} we obtain (\ref{eq_some})  in this case as well. 
 The result follows by noting that 
 $$\det \bp^{(2E,2E^{\wt\tau})}_{/i_{(2E,2E^{\wt\tau})}}\cong \det \bp^{(E,E^{\wt\tau})}_{/i_{(E,E^{\wt\tau})}}\otimes\det \bp^{(E,E^{\wt\tau})}_{/i_{(E,E^{\wt\tau})}}$$
  is canonically oriented and thus its first Stiefel-Whitney class vanishes. 
 \end{proof}
\begin{remark}
If $L$ is spin and all Maslov indexes are divisible by 4, we can take for $(E,F)$ the 0-dimensional   bundle pair.
\end{remark}

\begin{example}\label{ex_cp} The conditions  of Theorem \ref{cor_rs} are satisfied for the fixed locus of every anti-symplectic involution on    $\CP^{4n-1}$  or on a Calabi-Yau manifold,  provided it is spin. In particular, they are satisfied for $\R \mathbb{P}^{4n-1}$. Theorem \ref{cor_rs}  also applies to $(Q\times Q, \omega\oplus -\omega)$ and $L$ being the graph of a symplectomorphism $f$ on a symplectic manifold  $(Q,\omega)$ with $w_2(Q)=0$; in this case $L$ is the fixed locus of the anti-symplectic involution $(x, y)\mapsto(f^{-1}(y),f(x))$. The conditions are also satisfied for $(\CP^{4n+1},\R\mathbb{P}^{4n+1})$ with $E$ being the hyperplane bundle over $\CP^{4n+1}$.  They also hold for $\CP^1\times \CP^1$ with the standard involution, since  
$$b=\bar b\quad\Rightarrow \quad b=2b'\in H_2(\CP^1\times \CP^1,\R\mathbb{P}^1\times \R\mathbb{P}^1)$$ and thus the Maslov index of $b$ is 0 (mod 4).  This theorem also applies for the complete intersections $X_{n;{\bf a}}$ satisfying the conditions of Corollary~\ref{cor_cp}.
\end{example}

Let $\mathcal{O}(a)$ be the $a$-fold product of the hyperplane line bundle over $\CP^n$ and $\mathcal{O}^\R(a)$ the $a$-fold product of the hyperplane line bundle over $\RP^n$. Define  
    $$V_{n;\bf{a}}=\mathcal{O}(a_1)\oplus\dots\oplus\mathcal{O}(a_m)\quad\text{and}\quad V^\R_{n;\bf{a}}=\mathcal{O}^\R(a_1)\oplus\dots\oplus\mathcal{O}^\R(a_m).$$   
If $s\co\CP^n\ri V_{n;\bf{a}}$ is a  generic real section, $s^{-1}(0)$ is a complete intersection $X_{n;\bf{a}}$;  it is  Calabi-Yau if $n+1=\sum a_i$ and  Fano if $n+1>\sum a_i$. The section $s$ induces a section $\wt s$ of the bundle $$\ind \bp^{(V_{n;\bf{a}} ,V^\R_{n;\bf{a}})}_{/i_{(V_{n;\bf{a}},V^\R_{n;\bf{a}})}}\ri\M_{k,l+1}^*(\CP^n, A)$$ and $\wt s^{-1}(0)=\M_{k,l+1}^*(X_{n;\bf{a}}, A)$. The result of Corollary \ref{cor_cp} implies that under its conditions, $$w_1(\M_{k,l+1}^*(X_{n;\bf{a}}, A))=\mk{f}^*w_1(\M_{k,l+1})+U^{\spcheck}.$$ Moreover, if the open Gromov-Witten invariants are defined for such $X_{n;\bf{a}}$, they can be computed by an Euler class integration as in the classical case; see Proposition \ref{prop_ci}.

    \begin{cor}\label{cor_cp} If
    $n+1\equiv\sum a_i $ mod $2$ and $\sum a_i(a_i-1)\equiv0$ mod $4$, then
$$w_1(\widetilde{\mathcal{M}}_{k,l+1}^*(\CP^n,A))= w_1(\det \bp^{(V_{n;\bf{a}} ,V^\R_{n;\bf{a}})}_{/i_{(V_{n;\bf{a}},V^\R_{n;\bf{a}})}})+\mathfrak{f}^*w_1(\widetilde{\mathcal{M}}_{k,l+1})+U^{\spcheck}. $$
 In particular,
  $$w_1(\widetilde{\mathcal{M}}_{0,l+1}^*(\CP^n,A))=
             w_1(\det \bp^{(V_{n;\bf{a}} ,V^\R_{n;\bf{a}})}_{/i_{(V_{n;\bf{a}},V^\R_{n;\bf{a}})}}).$$ 
\end{cor}

 \begin{proof} 
The conditions $n+1\equiv\sum a_i$ mod $2$ and $\sum a_i( a_i-1)\equiv 0$ mod $4$ ensure that we can apply Corollary~\ref{prop_cp} with 
$$(E, F)=\begin{cases} (V_{n;\bf{a}} ,V^\R_{n;\bf{a}}),&
\text{if}~ \sum a_i+n+1\equiv0~ \text{mod}~ 4\\ 
(V_{n;\bf{a}}\oplus 2\mathcal{O}(1) ,V^\R_{n;\bf{a}}\oplus 2\mathcal{O}^\R(1)),& \text{if}~  \sum a_i+n+1\equiv2~ \text{mod}~ 4.
\end{cases}
$$
 In the latter case, we also note that $$w_1(\det \bp^{(V_{n;\bf{a}}\oplus 2\mathcal{O}(1) ,V^\R_{n;\bf{a}}\oplus 2\mathcal{O}^\R(1))}_{/i_{(V_{n;\bf{a}}\oplus 2\mathcal{O}(1) ,V^\R_{n;\bf{a}}\oplus 2\mathcal{O}^\R(1)})} )=w_1(\det \bp^{(V_{n;\bf{a}}  ,V^\R_{n;\bf{a}})}_{/i_{(V_{n;\bf{a}}  ,V^\R_{n;\bf{a}} )}}),$$ similarly to  the proof of Theorem \ref{cor_rs}. 
\end{proof}

\section{Open Gromov-Witten disk invariants}\label{sec_ogw}

In the first part of this section, we define   open Gromov-Witten invariants for  strongly semi-positive $\tau$-orientable manifolds. They provide a favourable environment in which  we can define a Gromov-Witten type pseudocycle using the more classical approach of $(J,\nu)$-holomorphic maps.
  As mentioned in Section~\ref{sec_intro}, one can go beyond the strongly semi-positive case by using  more sophisticated methods to achieve transversality such as Kuranishi structures \cite{FOOO, FO, LT}, polyfolds \cite{HWZ}, or systematic stabilizations~\cite{IP2}.  
In the second part, we describe the dependence of the invariants on the orienting choices and     define a refined invariant in the case when the fixed loci of the real maps represent more than one homology class in $H_1(L;\Z_2)$.

\subsection{Definition of invariants }

\begin{definition}\label{def_ssp}
A symplectic manifold $(M,\omega)$ is called 
\begin{enumerate}[label=(\arabic*), leftmargin=*] 
  \item \textsf{semi-positive} if $c_1(A)\geq 0 $ for all classes $A\in\pi_2(M)$ such that  $\omega (A)>0$ and $2c_1(A)\geq 6 -\dim(M)$,
  \item \textsf{strongly semi-positive} if $c_1(A)> 0 $ for all classes $A\in\pi_2(M)$ such that $\omega (A)>0$ and $2c_1(A)\geq 4 -\dim(M)$.
\end{enumerate}  
\end{definition}

In particular, all the relevant first Chern classes are strictly greater than zero in the strongly semi-positive case.  One defines   the classical Gromov-Witten invariants in the semi-positive case  by considering $(J,\nu)$-holomorphic maps and reducing the multiply-covered maps: the images of the latter  under the evaluation maps are contained in the image of the  reduced space, which has codimension at least two due to the semi-positive condition. This is sufficient for the image of the moduli space to define a pseudocycle. Similarly, under the strongly semi-positive condition, the image of the $\tau$-multiply covered maps is contained in a subspace of  codimension at least two, and thus  the image of $\widetilde{\mathcal{M}}_{k,l+1}(M,A)$ defines a pseudocycle.

\begin{definition}\label{def_adm} Let $(M,\omega)$ be a compact symplectic manifold with an anti-symplectic involution $\tau$ with non-empty fixed locus $M^\tau=L$. We say $(M,\omega,\tau)$ is \textsf{admissible} if
 it is 
 $\tau$-orientable and $M$ is strongly semi-positive.
   \end{definition}
 
The strongly semi-positive condition ensures that the image of the 
multiply-\!\!\! covered maps is of sufficiently large codimension. The $\tau$-orientable   condition implies that the first Stiefel-Whitney class of the moduli space of $\tau$-semi-simple maps is equal to $\mk f^*w_1(\M_{k,l+1})+~U^{\spcheck}$; see Theorem \ref{cor_rs}.  In particular, if there are no boundary marked points,   the moduli space $\M_{0,l+1}^*(M,A)$ is orientable.

\begin{lem} \label{lem_orm} If $(M,\omega, \tau)$ is $\tau$-orientable, a choice of a $\tau$-orienting structure determines   an orientation of the relative determinant bundle of (\ref{eq_indbl}), denoted by $\det D^{(TM,TL)}_{/i_{(TM,TL)}}$, over every point in   $\M_{k,l+1}^*(M,A)-U$; this orientation  varies continuously with the point.
\end{lem}

\begin{proof} Let $(E,\wt\tau)$ be as in Definition \ref{def_tori}.
 A choice of  a spin structure on the bundle $2E^{\wt\tau}\oplus TL$ induces an orientation on  the   determinant bundle 
 $$
 \det \bp^{(2E\oplus TM, 2E^{\wt \tau}\oplus TL)}\ri \wt{\mk{M}}^{*}_{k,l+1}(M,b)\qquad \forall b\in \H_2(M,L);
 $$ 
 see \cite[Proposition 8.1.4]{FOOO}. If all Maslov indices are divisible by 4, this orientation is continuous across all codimension-one strata except $U$; see Corollary \ref{cor_det}.
  If not all Maslov indices are divisible by 4, 
  we twist the orientation of
  $\det \bp^{(2E\oplus TM, 2E^{\wt \tau}\oplus TL)}$ over $\wt{\mk{M}}^{*}_{k,l+1}(M,b)$
  using the representatives $b_i$ in Definition \ref{def_tori} as follows.
  There is a unique $b_i$ such that $b=b_i-b'+\bar b'$ for some $b'\in H_2(M,L;\Z)$. We twist the orientation by $(-1)^{\frac{\mu(b')}{2}}$; this is well-defined since the Maslov indices of two such $b'$ differ by a multiple of 4 under our assumptions. 
   This makes the orientation of $\det \bp^{(2E\oplus TM, 2E^{\wt \tau}\oplus TL)}_{/i_{(2E\oplus TM, 2E^{\wt \tau}\oplus TL)}}$ continuous over $\M^{*}_{k,l+1}(M,A)-U$ in this case as well. Together     
   with the canonical orientation of
  $$\det \bp_{/i_{(2 {E},2 {E}^{\wt\tau})}}^{(2 {E},2 {E}^{\wt\tau})}\cong
  \det \bp_{/i_{( {E}, {E}^{\wt\tau})}}^{( {E}, {E}^{\wt\tau})}\otimes
  \det \bp_{/i_{( {E}, {E}^{\wt\tau})}}^{( {E}, {E}^{\wt\tau})},
  $$ this induces an orientation on $\det D^{(TM,TL)}_{/i_{(TM,TL)}}\cong \det \bp^{(TM,TL)}_{/i_{(TM,TL)}}$ over $\M_{k,l+1}^*(M,A)-U$ via the isomorphism
  $$
  \det \bp^{(2E\oplus TM, 2E^{\wt \tau}\oplus TL)}_{/i_{(2E\oplus TM, 2E^{\wt \tau}\oplus TL)}}\cong \det \bp_{/i_{(2 {E},2 {E}^{\wt\tau})}}^{(2 {E},2 {E}^{\wt\tau})}\otimes\det \bp^{(TM,TL)}_{/i_{(TM,TL)}}.
  $$ 
  This establishes the proof.
\end{proof}

\begin{proof}[{\bf \emph{Proof of Theorem \ref{thm_ogw0}}}] \ref{it_ori}
  By Lemma \ref{lem_orm},  a choice of a $\tau$-orienting structure induces an orientation on $\det D^{(TM,TL)}_{/i_{(TM,TL)}}\ri \M_{0,l+1}^*(M,A)$. 
  By Proposition~\ref{cor_dmor}, the moduli space of domains $\M_{0,l+1}$ is orientable  and has a canonical orientation. Together they induce an orientation on the moduli space via 
    \begin{equation}\label{eq_ciso}
    \Lt \M_{0,l+1}^*(M,A)\cong \det D^{(TM,TL)}_{/i_{(TM,TL)}}\otimes \mk f^* \Lt \M_{0,l+1}.
    \end{equation}
\ref{it_inv} The regularity of the pair $(J,\nu)$ implies that we have transversality for all  maps with stable domains and none of them is $\tau$-multiply covered. We have 
$$\M_{0,l+1}(M,A)-\M_{0,l+1}^*(M,A)= \{\text{$\tau$-multiply covered maps}\}$$ 
and 
$$\ov{\M_{0,l+1}^*(M,A)}\subset \M_{0,l+1}(M,A).$$
 Thus,  
 $$\ov{\M_{0,l+1}^*(M,A)}-\M_{0,l+1}^*(M,A)\subset\{\text{$\tau$-multiply covered maps}\}.$$
 The image of the $\tau$-multiply covered maps under $\ev$ or $\ev\times \mk{f}$ is contained in the image of the reduced maps and has codimension at least two by the strongly semi-positive assumption.  
 By Theorem \ref{nb}\ref{it_nb}, $\M_{0,l+1}^*(M,A)$ has no boundary for $A\in\mathcal{A}$. Thus, by   \cite[Lemma 3.5]{Zin06}, the image of $\M_{0,l+1}(M,A)$ under $\ev$ and $\ev\times \mk{f}$ defines a homology class.  

 Two regular pairs $(J_1,\nu_1)$ and $(J_2,\nu_2)$   in $\mathcal{J}_\R$  can be connected by a path. For a generic such path, the image of the boundary of  the set of $\tau$-semi-simple maps in the universal moduli space over this path is again contained in a space of codimension at least two; this follows  by the strongly semi-positive condition which ensures that all $J_t$-holomorphic maps along this path have strictly positive Chern/Maslov class; see \cite[Lemma 6.4.4]{MS}. Thus, by   \cite[Lemma 3.6]{Zin06},  the two homology classes defined by $(J_1,\nu_1)$ and $(J_2,\nu_2)$ are  the same. A similar cobordism argument holds for a strongly semi-positive deformation of~$\omega$.
\end{proof}

\begin{example}\label{ex_crcex}
Let $(M,L)=(\CP^{2n-1},\R \mathbb{P}^{2n-1})$.   All odd-degree classes   $A$ belong to~$\mathcal{A}$. 
The moduli space $\M_{0,l+1}$ is orientable, and hence we can use the dual of its fundamental class as the constraint $h^{DM}$. 
Since  $\M_{0,l+1}(\CP^{2n-1},A)$ is isomorphic to the moduli space of real sphere maps,   the number $\OGW_{A,0,l+1}(\mathbf{h})$, defined immediately after Theorem~\ref{thm_ogw0},  can be interpreted in this case as the number of real spheres in the class~$A$ passing through the Poincare duals of $h^M_1,..,h^M_{l+1}$. In particular, if we take $A$ to be the class of a line, $l=1$, and $\mathbf{h}=(\text{pt}^{\spcheck}, H^{\spcheck}, [\M_{0,2}]^{\spcheck})$, where $H$ is the homology class of  a hyperplane, then $\OGW_{A,0,2}(\mathbf{h})$ counts the number of real spheres   passing through a complex point and intersecting a hyperplane. The real condition forces the curve to also pass through the complex conjugate of the point, and hence there is only one such curve. Thus, the number 
$$|\OGW^{\CP^{2n-1}}_{\text{line},0,2}(\text{pt}^{\spcheck}, H^{\spcheck}, [\M_{0,2}]^{\spcheck})|=1.$$ Note that $z_0$ must go to $\text{pt}$ and {\it not}    its conjugate.  A complete recursion determining the rest of the invariants is obtained in \cite{GZ2}.
\end{example}

   If $k>0$, we define the invariant as a signed count of maps passing through prescribed constraints still under the assumption that  $(M, \omega,\tau)$ is admissible. 
   In order to define the sign at such maps, we orient  $\det D$ over $\M_{k,l+1}^*(M,A)-U$ as in Lemma \ref{lem_orm}. We define a local orientation at a smooth point $C$ in the moduli space of domains using the canonical orientation of the component $\wt{\mk{M}}_{k,l+1}$ the curve C belongs to. We use this local orientation and the orientation on the determinant bundle to induce local orientation on the smooth points in the moduli space of maps which lie over $C$ via the isomorphism~(\ref{eq_ciso}). 
 Let $A_i, B_j $, and $\Gamma$, for $i=1,\dots, l+1, j=1,\dots, k,$ be oriented manifolds representing   homology classes in $H_*(M;\Z), H_*(L;\Z)$, and $H_*(\M_{k,l+1};\Z)$, respectively, and let $\alpha_i, \beta_j$, and $\gamma$ be the Poincare duals of these classes. Let $|\alpha|$ denotes the degree of a cohomology class $\alpha$ and suppose 
 \begin{equation}\label{eq_dimc}
 \sum_i|\alpha_i|+\sum_j|\beta_j|+|\gamma| =\dim \M_{k,l+1}(M,A).
 \end{equation}
  Under generic conditions,   the image $(\ev\times\mk{f})(\M_{k,l+1}(M,A))$ intersects the image of the product $ \prod_iA_i\times\prod_jB_j\times \Gamma$ at finitely many points whose preimages in $\M_{k,l+1}(M,A)$ are maps with   smooth domains.  For such a preimage point $u\in  \M_{k,l+1}(M,A)$, denote by $\mk s(u) $ the sign   of the isomorphism at $u$
 $$
 T(\M_{k,l+1}(M,A))\times\prod_iT(A_i)\times\prod_jT(B_j)\times T(\Gamma)\cong T(M^{l+1}\times L^k\times \M_{k,l+1}),
 $$ 
where the two sides are oriented using the local orientations described above.
 This sign is independent of the choice of local orientation on the moduli space of domains.

 \begin{proof}[{\bf \emph{Proof of Theorem \ref{thm_ogwk}}}]
  Let $(J_1,\nu_1)$ and $(J_2,\nu_2)$ be two regular pairs in $\mathcal{J}_\R$ and $(J_t,\nu_t)$ be a generic path between them. As in the proof of Theorem \ref{thm_ogw0}, the image of the $\tau$-multiply covered maps is of   codimension at least 2 and in a generic path they are  avoided. Since $k>0$, there is no codimension-one sphere bubbling. 
  Thus, along the path $\{J_t\}$ the intersection points will form a one-dimensional bordism, crossing a finite number of codimension-one strata  of the moduli space.  We define the local orientations on the precompactified universal moduli space as we did for a fixed $J$ and orient the corresponding pieces of the bordism   by the surjection  
 \begin{align*}
 T(\M_{A,k,l+1}(M,A; \{J_t\}))\times\prod_iT(A_i)&\times\prod_jT(B_j)\times T(\Gamma)\\&\ri T(M^{l+1}\times L^k\times \M_{k,l+1});
 \end{align*}
 the  restriction of this surjection  to $J_0$ gives   the sign $-\mk s(u;J_0)$ and  the restriction to   $J_1$ gives $\mk s(u;J_1)$. We first show that if we do not cross the stratum $U$, the described orientation is continuous across the codimension-one strata and thus the count at $J_0$ equals the count at $J_1$;  the orientation is clearly continuous if we do not cross codimension-one strata. If we cross a codimension-one stratum where  the bubbled domain is stable, the changes of the local orientations on the moduli spaces of maps and domains are the same; see the proof of Theorem~\ref {gc}.   There is no change in the local orientation on the moduli space of maps if the domain of the second bubble is unstable, but does not have any boundary marked points, which means the orientation is continuous across such a stratum as well. \\

If there is a single boundary marked point on the second bubble,  the local orientation on the moduli space of maps changes, whereas on the moduli space of domains it does not; this is the stratum $U$ and the orientation is not continuous across it. 
 We show that under the assumptions   of the theorem  $U$ is never crossed. Suppose that in a one parameter family the cut-down moduli space bubbles into classes $b_1$ and $b_2$ with  a single boundary marked point on the second bubble with image on $B_{j_0}$. 
 By a dimension count, this happens  only~if 
 \begin{align*}
 \mu(b_1)+n-3+k+2(l+1)&\geq \sum_i |\alpha_i|+\sum_{j\neq j_0}|\beta_j|+|\gamma|\quad\text{and}\\ \mu(b_2)+n-3+2&\geq |\beta_{j_0}|.
 \end{align*}
   Adding $|\beta_{j_0}|$ to the first inequality and using (\ref{eq_dimc}), we get  
 \begin{gather*}
 \mu(b_1)+n-3+k+2(l+1)+|\beta_{j_0}|\geq \mu(b)+n-3+k+2(l+1). 
 \end{gather*}
Since $\mu(b)=\mu(b_1)+\mu(b_2)$, this is equivalent to $|\beta_{j_0}|\geq \mu(b_2).$ 
  However,  $$\mu(b_2)\geq c_{\min}>|\beta_{j_0}|,$$ by assumption.  Thus, $U$ is never crossed.
  A similar cobordism argument holds for a strongly semi-positive deformation of $\omega$ and for a change of the representatives of the homology classes.
  \end{proof}

If the moduli space of domains $\M_{k,l+1}$ is orientable, we can use  its fundamental class as a (trivial) constraint and   obtain a count of the maps passing through constraints only from $M$ and $L$, i.e.~a primary Gromov-Witten invariant. For example, 
$$|\OGW^{\CP^{2n-1}}_{\text{line},2,1}(\text{pt}^{\spcheck}, \text{pt}^{\spcheck}, H^{\spcheck}, [\M_{2,1}]^{\spcheck})|=1,$$
where pt$^{\spcheck}$ is the Poincare dual of a point in $\RP^{2n-1}$; see Example \ref{ex_crcex}. When $\M_{k,l+1}$ is not orientable, however, $H_{\text{top}}(\M_{k,l+1};\Z)=0$, and we have to use different constraints in order to have an invariant count. This can be interpreted as a restriction on the domains of the maps.  The next example demonstrates the non-triviality of the invariants in this case.

\begin{example}\label{bpts_ex} Let $\Gamma$ be the stratum of $\M_{1,l+1}$, $l\!>\!1$, composed of nodal domains  having a disk with one boundary marked point attached to a sphere with $l\!+\!1$ marked points, all decorated with a $+$. This is a closed oriented codimension 2 stratum of $\M_{1,l+1}$ and thus represents a $\Z$-homology class. Let $\text{cp}^{\spcheck}$ be the dual of a point and~$H^{\spcheck}$ the dual of a hyperplane in $\CP^{2n-1}$ and $\text{rp}^{\spcheck}$ the dual of a point in $\RP^{2n-1}$. Let $\beta_1,\dots, \beta_{l}\in H^*(\CP^{2n-1})$ be such that $\sum_{i=1}^l|\beta_i|= 4nd\!-\!4\!+\!2l$. Then, 
$$
|\OGW^{\CP^{2n-1}}_{2d, 1, l+1}(\text{rp}^{\spcheck}, \beta_1,\dots,\beta_l , H^{\spcheck}, [\Gamma]^{\spcheck})| = \text{GW}^{\CP^{2n-1}}_{d, l+2}(\beta_1,\dots, \beta_l, H^{\spcheck}, \text{cp}^{\spcheck}),
$$
where the number on the right-hand side of the equality is the classical Gromov-Witten invariant of $\CP^{2n-1}$ in degree $d$ passing through a point, a hyperplane, and the duals of $\beta_1, \dots, \beta_l$. This equality holds since $\Gamma$ restricts the domains of the maps in such a way that only maps which are constant on the disk component will contribute; all other possibilities vanish for dimensional reasons.
\end{example}

In order to define an invariant count of maps without restricting the domains,   one has to
show
that in a one-parameter family the cut-down moduli space does not cross boundary divisors which contribute to the first Stiefel-Whitney class of $\M_{k,l+1}(M,A)$.
This 
 is possible in dimension 4 and 6 as shown in 
 \cite{Cho, Sol}  but the approach does not extend to higher dimensions;  we  illustrate it in the following example.

\begin{example}[{\cite[Section 3]{Cho}, \cite[Section 6]{Sol}}]\label{ex_Jake} Suppose $(M,L)=(\CP^3, \RP^3)$,  the class $A$ has degree $d$, and     only real point constraints at the boundary marked points and hyperplane constraints at the interior points are used. The dimension formula implies that the number of boundary marked points $k$ must equal~$2d$. Suppose in a one-parameter family the cut-down moduli space bubbles into classes of degrees $d_1$ and $d_2$ with $k_1$ and $k_2$ boundary marked points, respectively. By a dimension count, this stratum appears only if $4d_1+1\geq 2k_1$ and $4d_2+1\geq 2k_2$.
 We add $2k_2$ to the first inequality to obtain
\[ 4d_1+1+2k_2\geq 2k_1+2k_2=2k=4d\Rightarrow  2k_2\geq 4d_2-1.\]
This together with the second inequality implies that $k_2=2d_2$ is even, and so should be $k_1$. Thus, we never cross strata which contribute to the first Stiefel-Whitney class of $\M_{k,l+1}(M,A)$. This computation, originally appearing in \cite[Section 3]{Cho} and \cite[Section 6]{Sol}, relies heavily on the fact that the dimension of the Lagrangian is small and the degrees of the cohomology constraints are big, which gives a strong relation between the possible combinations of Maslov indices and marked point splits we encounter in a one-parameter family. This relation quickly   weakens with  the increase of the dimension. 
\end{example}
 
\begin{prop} \label{prop_ci} Let $X_{n;\bf{a}}$ be a Fano complete intersection satisfying the conditions of  Corollary \ref{cor_cp},   $A\in\mathcal{A}$, and  $${\bf h}=(h_1,\dots, h_{l+1}, h^{DM})\in H^*(\CP^{n};\Z)^{\oplus (l+1)}\oplus H^*(\M_{0,l+1};\Z).$$ Then,
\[
\OGW^{X_{n;\bf{a}}}_{A,0,l+1}({\bf h})= \int_{[\M_{0,l+1}(\CP^{n},A)]} e(\ind \bp^{(V_{n;\bf{a}} ,V^\R_{n;\bf{a}})}_{/i_{(V_{n;\bf{a}},V^\R_{n;\bf{a}})}})\wedge h_1\wedge\dots\wedge h_{l+1}\wedge h^{DM}.
\]

\end{prop}
\begin{proof}
By \cite[Proposition 11]{psw}, $\M_{0,l+1}(\CP^{n},A)$ is a smooth closed orbifold and 
$$
\ind \bp^{(V_{n;\bf{a}} ,V^\R_{n;\bf{a}})}_{/i_{(V_{n;\bf{a}},V^\R_{n;\bf{a}})}}\ri\M_{0,l+1}(\CP^{n},A)
$$
 is a smooth orbibundle; our construction and the construction of \cite[Proposition 11]{psw} differ only by the rule  specifying which bubble is conjugated, which does not affect the argument. By \cite[Corollary 1.10]{Geo1}, there is a canonical isomorphism 
$$\Lt  \wt{\mk{M}}_{0,l+1}(\CP^n,b)\cong \det \bp^{(V_{n;\bf{a}} ,V^\R_{n;\bf{a}})}_{/i_{(V_{n;\bf{a}},V^\R_{n;\bf{a}})}}|_{\wt{\mk{M}}_{0,l+1}(M,b)}.
$$
If $n+1\equiv\sum a_i$ mod 4, $w_i(\R\mathbb{P}^n)=w_i(V^\R_{n;\bf{a}})$ for $i=1,2$   and the relative signs of the gluing maps on the bundle and on the moduli space are equal by
   Corollaries \ref{cor_det} and~\ref{cor_ms}.
    Thus, the canonical isomorphism extends over $\M_{0,l+1}(\CP^n, A)$. 
If $n+1\equiv \sum a_i +2$ mod 4, we apply the same argument to $V_{n;\bf{a}}\oplus 2\mathcal{O}(1)$ and note that $\det \bp^{( 2\mathcal{O}(1),2\mathcal{O}^\R(1))}_{/i_{( 2\mathcal{O}(1),2\mathcal{O}^\R(1))}}$ is canonically oriented. Thus, there is a canonical isomorphism 
$$
\Lt  \M_{0,l+1}(\CP^n,A)\cong \det \bp^{(V_{n;\bf{a}} ,V^\R_{n;\bf{a}})}_{/i_{(V_{n;\bf{a}},V^\R_{n;\bf{a}})}}
$$
and   the Euler class 
$$e(\ind \bp^{(V_{n;\bf{a}} ,V^\R_{n;\bf{a}})}_{/i_{(V_{n;\bf{a}},V^\R_{n;\bf{a}})}})\in H^{\text{top}}(\M_{0,l+1}(\CP^n,A); \mathcal{Z}_{w_1(\M_{0,l+1}(\CP^n,A))}),
$$
 where $\mathcal{Z}_{w_1(\M_{0,l+1}(\CP^n,A))}$ is the local system of orientations on the moduli space. This implies the integral is well-defined and the fact that it   gives the open Gromov-Witten invariant   follows as in the classical case and in the proof of \cite[Theorem 3]{psw}.
\end{proof}

\subsection{Dependence on     orienting choices }\label{ssec_doc}
 
The signed counts of   maps arising from Theorems \ref{thm_ogw0} and \ref{thm_ogwk} are independent of the choice of a (local) orientation on the space of domains (and moreover, there is a canonical choice).
 Lemma \ref{lem_orm} orients  $\det D\ri \M_{k,l+1}^*(M,A)-U$ using a choice of a spin structure on $2 E^{\wt\tau}\oplus TL$  and  if  not all Maslov indices are divisible by $4$, a choice of representatives  $b_i\in H_2(M,L;\Z)$ for the elements of $H_2(M,L;\Z)/\text{Im}(\id+\tau_*)$.  \\
 
 Suppose all Maslov indices of $(2 E\oplus TM, 2 E^{\wt\tau}\oplus TL)$ are divisible by 4 and thus the (local) orientation on the moduli space is fixed by a choice of a spin structure on $2 E^{\wt\tau}\oplus TL$. 
  Example~\ref{ex_ntq} shows that the invariants depend non-trivially on this choice. The dependence can be seen as follows.
   Let $\mathcal{B}\subset H_1(L;\Z_2)$ be the set of classes which the fixed loci of the real maps in the class $A\in \H_2(M)$ represent. For every choice of a class $q\in \mathcal{B}$, we define a refined invariant $\OGW^q_{A,k,l+1}$ counting only those maps whose fixed locus  represent the class $q$. The proof that these numbers are invariant is exactly the same as before and 
$$
\OGW_{A,k,l+1}=\sum_{q\in \mathcal{B}} \OGW^q_{A,k,l+1}.
$$
Changing the spin structure negates   $\OGW^q_{A,k,l+1}$ if the change is non-trivial over~$q$. This implies that $\OGW_{A,k,l+1}$ may change even in the absolute value if $|\mathcal{B}|>1$. However, the invariant 
 $$
\sum_{q\in \mathcal{B}} |\OGW^q_{A,k,l+1}|
$$
is independent of this choice. Moreover, for every choice  of a spin structure
$$
\sum_{q\in \mathcal{B}} |\OGW^q_{A,k,l+1}|\geq \OGW_{A,k,l+1}.
$$
This refinement holds for the Welschinger's invariant \cite{Wl, Wel} as well, where we can sum over the absolute value of the invariant for a fixed $q$.

\begin{example}\label{ex_ntq}  
Let $M=T^2\times \CP^{4n-1}$ and $\tau(s,t, z)=(-s , t, \bar z)$, where $s,t\in \R/2\pi\Z$ are  angular coordinates. The fixed locus, 
$$L = \{0,\pi\}\times S^1\times \R\mathbb{P}^{4n-1},$$ is disconnected. 
Let $A=\text{pt}\times\text{pt}\times \text{line}\in H_2(M)$, $H_1= S^1\times\text{pt}\times H\in H_{8n-3}(M)$, where $H$ is the homology class of a hyperplane, and $H_2= T^2\times \text{pt}\in H_{2}(M)$. There are exactly two real curves passing through $H_1$ and $H_2$ in the class $A$. The sign with which we count these curves is determined by a choice of a spin structure on the (disconnected) Lagrangian $L$, and we are free to choose a different one on each component. Thus,  $\OGW_{A,0,2}= 0$ for some choices   and  $|\OGW_{A,0,2}|=2$ for the other choices. The cardinality of the set $\mathcal{B}$ in this case is 2, with representatives $\{0\}\times\text{pt}\times \R\mathbb{P}^1$ and $\{\pi\}\times\text{pt}\times \R\mathbb{P}^1$, and 
$$ \sum_{q\in B} |\OGW^q_{A,0,2}(M,A)|=2,$$ regardless of the choice of  a spin structure.
\end{example}

If not all Maslov indices of $(2 E\oplus TM,2 E^{\wt\tau}\oplus TL)$ are divisible by 4, the invariants depend on the choice of representatives $b_i\in H_2(M,L;\Z)$ in Definition~\ref{def_tori} used to determine the orientation.

\end{document}